\documentclass[12pt,letterpaper]{article}

\usepackage{natbib,hyperref}

\usepackage[ left=1in, top=1in, right=1in, bottom=1in]{geometry}
\usepackage{graphicx,bm,colonequals,amsmath,amssymb,url,xcolor,bbm}
\usepackage{array,tabularx,multirow}
\usepackage{enumitem}
\usepackage[font={footnotesize}]{caption,subcaption}
\usepackage[utf8]{inputenc}
\usepackage{enumitem}
\usepackage{soul} 
\usepackage{placeins} 
\usepackage[normalem]{ulem}

\usepackage{titling} 
\predate{}
\postdate{}

\graphicspath{{plots/}}

\usepackage{mathtools}
\mathtoolsset{showonlyrefs} 

\bibpunct[, ]{(}{)}{;}{a}{,}{,}
   
\usepackage{amsthm}
\newtheoremstyle{propstyle} 
    {2mm}                    
    {1mm}                    
    {\itshape}                   
    {}                           
    {\scshape}                   
    {.}                          
    {.5em}                       
    {}  
\theoremstyle{propstyle}
\newtheorem{proposition}{Proposition}
\theoremstyle{propstyle}

\theoremstyle{propstyle}
\newtheorem{lemma}{Lemma}
\theoremstyle{propstyle}

\theoremstyle{propstyle}

\theoremstyle{propstyle}

\makeatletter
\renewcommand{\paragraph}{%
  \@startsection{paragraph}{4}%
  {\z@}{2ex \@plus 1ex \@minus .2ex}{-1em}%
  {\normalfont\normalsize\bfseries}%
}
\makeatother

\DeclareMathAlphabet\mathbfcal{OMS}{cmsy}{b}{n}

\newcommand{\ba}{\mathbf{a}}

\newcommand{\bs}{\mathbf{s}}

\newcommand{\bx}{\mathbf{x}}
\newcommand{\by}{\mathbf{y}}

\newcommand{\bz}{\mathbf{z}}

\newcommand{\bA}{\mathbf{A}}

\newcommand{\bI}{\mathbf{I}}

\newcommand{\bU}{\mathbf{U}}

\newcommand{\bK}{\mathbf{K}}

\newcommand{\bfzero}{\mathbf{0}}

\newcommand{\bftheta}{\bm{\theta}}

\newcommand{\bfomega}{\bm{\omega}}

\newcommand{\bfTheta}{\bm{\Theta}}

\newcommand{\GP}{\mathcal{GP}}

\newcommand{\trace}{\text{tr}}

\newcommand{\order}{\mathcal{O}}
\newcommand{\knuth}{\Omega}
\newcommand{\normal}{\mathcal{N}}

\DeclareMathOperator*{\argmax}{arg\,max}
\newcommand{\KL}{\text{KL}}

\newcommand{\inputdomain}{\mathcal{X}}

\newcommand{\locs}{\mathcal{S}}
\newcommand{\dens}{p}
\newcommand{\adens}{\hat p}
\newcommand{\boundary}{\Lambda}

\title{Asymptotic properties of Vecchia approximation for Gaussian processes}

\author{Myeongjong Kang\thanks{Department of Statistics, Texas A\&M University} \and Florian Sch\"afer\thanks{School of Computational Science and Engineering  Georgia Institute of Technology} \and Joseph Guinness\thanks{Department of Statistics and Data Science, Cornell University} \and Matthias Katzfuss\thanks{Department of Statistics, University of Wisconsin--Madison (\texttt{katzfuss@gmail.com})}}

\date{}

\begin{document}

\maketitle

\begin{abstract}
Vecchia approximation has been widely used to accurately scale Gaussian-process (GP) inference to large datasets, by expressing the joint density as a product of conditional densities with small conditioning sets. We study fixed-domain asymptotic properties of Vecchia-based GP inference for a large class of covariance functions (including Mat\'ern covariances) with boundary conditioning. In this setting, we establish that consistency and asymptotic normality of maximum exact-likelihood estimators imply those of maximum Vecchia-likelihood estimators, and that exact GP prediction can be approximated accurately by Vecchia GP prediction, given that the size of conditioning sets grows polylogarithmically with the data size. Hence, Vecchia-based inference with quasilinear complexity is asymptotically equivalent to exact GP inference with cubic complexity. This also provides a general new result on the screening effect. Our findings are illustrated by numerical experiments, which also show that Vecchia approximation can be more accurate than alternative approaches such as covariance tapering and reduced-rank approximations.
\end{abstract}



\section{Introduction \label{sec:intro}}

Gaussian processes (GPs) are popular models for functions, time series and spatial fields, with myriad application areas such as spatial statistics \citep[e.g.,][]{Stein1999,Banerjee2004,Cressie2011}, nonparametric regression and machine learning \citep[e.g.,][]{Rasmussen2006,Deisenroth2010}, the analysis of computer experiments \citep[e.g.,][]{Sacks1989,Kennedy2001,Gramacy2020}, and optimization \citep[e.g.,][]{Mockus1989,Jones1998,Frazier2018}. GP models provide flexible and interpretable tools for uncertainty quantification, as opposed to algorithmic approaches in machine learning. Further, graphical models stacking GPs as layers have been identified as promising tools for deep belief networks \citep[e.g.,][]{Damianou2013,Dunlop2018,Dutordoir2021}. 

Direct application of GPs is computationally intractable for large data sets, since evaluation of multivariate Gaussian densities generally incurs computational cost that is \emph{cubic} in the data size. To tackle this issue, various approximations have been proposed: Covariance tapering \citep[e.g.,][]{Furrer2006,Kaufman2008} imposes sparsity on covariance matrices by truncating long-range correlations, which implies that it, by definition, cannot capture large-scale variation. While asymptotic theory for tapering approximations is well developed \citep{Kaufman2008,Du2009,Shaby2010,Wang2011}, the time complexity of tapering with a fixed radius under fixed-domain asymptotics is still cubic in the data size. Reduced-rank approximation \citep[e.g.,][]{Higdon1998, Wikle1999, Quinonero2005, Banerjee2008, Cressie2008, Katzfuss2010} imposes low-rank structure and thus enjoys various computational advantages, but typically leads to poor approximations of fine-scale variation \citep[e.g.,][]{Stein2013a}. Spectral approximation \citep{Whittle1954,Dahlhaus1987} can produce fast and accurate approximations by leveraging the fast Fourier transform, but its use is limited to stationary models and to gridded data, although recent work has extended its use to incomplete grids \citep{Guinness2017,Guinness2017a}. The use of $\mathcal{H}$- and $\mathcal{H}^2$-matrix approaches \citep{Borm2007} can be also limited due to the challenges with construction of the hierarchy in high dimensions. Several approximations from spatial statistics \citep[e.g.,][]{Lindgren2011a,Nychka2012,Datta2016} have relied on limiting the number of nonzero entries per row of the precision matrix. However, the computational cost of algorithms involving sparse matrices depends on both the number and the locations of the nonzero entries in a complex way. Hence, the computational complexity for sparse precision matrices of size $n \times n$ is not linear in $n$, but rather $\order(n^{3/2})$ in two spatial dimensions, and at least $\order(n^2)$ in higher dimensions.

In recent years, an old idea by \citet{Vecchia1988}, the so-called Vecchia approximation, has become very popular in the spatial-statistics and machine-learning communities. The Vecchia approximation expresses the joint density as a product of conditional densities with reduced conditioning sets. 
Independently of Vecchia's work, \citet{Kaporin1990} derived the equivalent approximation as a sparse factorized approximate inverse preconditioner that is optimal with respect to the Kaporin condition number.  
\citet{Schafer2020} arrived at the same approximation as the minimizer of the Kullback-Leibler (KL) divergence to the exact GP model among all approximations that assume a specific sparsity pattern on the inverse Cholesky factor of the covariance matrix. According to \cite{Katzfuss2017a}, many existing GP approximations can be viewed as special cases of a Vecchia framework, including reduced-rank methods \citep{Quinonero2005, Banerjee2008, Finley2009} and approaches based on iterative partitioning of the domain such as the multi-resolution approximation \citep{Katzfuss2015,Katzfuss2017b} and its special cases \citep[e.g.,][]{Snelson2007,Sang2011a}. 
Vecchia approximations have been shown to perform well in many settings, including spatial, high-dimensional, nonstationary, and multivariate GPs \citep[e.g.,][]{Pardo1997,Eide2002,Stein2004,Datta2016,Guan2020,Katzfuss2018,Katzfuss2017a,Katzfuss2020,Schafer2020,Cao2022,Wu2022,Cao2023,Jimenez2022,Kang2021} and for solving general nonlinear partial differential equations \citep[e.g.,][]{Chen2023}.
 
There have been limited efforts devoted to providing theoretical understanding of Vecchia approximation and the required size of the complexity parameter (i.e., the size of conditioning sets) that trade off computational speed against accuracy. 
One of the earliest attempts to investigate large-sample properties of the Vecchia approach was made by \cite{Stein2004}. Based on the fact that the score function induced by Vecchia approximation is always unbiased, their work suggested that consistency and asymptotic normality of restricted maximum Vecchia-likelihood estimators can be studied using martingale limit theory \citep{Heyde1997,Hall2014}, although a proof was not provided.
\cite{Zhang2012} alluded to an unpublished result stating that the consistency for the maximum Vecchia-likelihood (MVL) estimators of a microergodic parameter holds if certain asymptotics of the conditional expectation given the conditioning set hold. 
Motivated by \cite{Du2009} and \cite{Zhang2012}, \cite{Zhang2021} provide fixed-domain asymptotic properties of Vecchia approximation, but their main results are based on strong assumptions on prediction errors that are only verified for GPs on a one-dimensional domain; also, this work did not provide explicit guidelines for conditioning-set sizes. \cite{Zhu2022} proposed a new ordering and conditioning approach for Vecchia approximations of GPs and provided extensive theoretical results regarding its approximation accuracy, but the applicability to infill asymptotics remains unclear as their theoretical justification depends on the decay properties of the GP's covariance function. \cite{Schafer2020} provided promising results for establishing asymptotic properties of Vecchia approximation, based on a series of works in numerical homogenization and operator adapted wavelets \citep{Malqvist2014, Kornhuber2016, Owhadi2017, Kornhuber2018, Owhadi2019, Schafer2017}, but it remains unclear how these results are related to consistency and asymptotic normality of MVL estimators of model parameters as well as asymptotic properties of GP prediction based on Vecchia approximation.
Most previous studies on asymptotic properties of Vecchia approximation have not thoroughly focused on similarities between asymptotic properties of exact GP and its Vecchia approximation, despite interesting results on asymptotic properties of exact GPs, including asymptotic normality of maximum exact-likelihood (ML) estimators for identifiable parameters of exact GPs under fixed-domain asymptotics \citep[e.g.,][]{Zhang2004,Anderes2010,Kaufman2013}.
Since the Vecchia likelihood generally yields a close approximation to the true GP likelihood for sufficiently large size of conditioning sets relative to $n$, it can be expected that asymptotic properties of Vecchia approximation depend on those of its exact GP.

In this paper, we provide theoretical evidence suggesting that consistency and asymptotic normality of ML estimators imply those of MVL estimators for GPs with covariance functions satisfying certain conditions motivated by the partial differential equation (PDE) literature, and prediction (or kriging) based on an exact GP can be asymptotically equivalent to that based on its Vecchia approximation, given that the size of conditioning sets grows polylogarithmically with the data size under the fixed-domain asymptotic framework. These results are supported by our numerical experiments, which also show that Vecchia approximation can be more accurate than alternative approaches such as covariance tapering and reduced-rank approximations for parameter estimation and GP prediction, and perform similarly as the exact GP.

The remainder of this document is organized as follows. In Section \ref{sec:vecchia}, we introduce a GP model and Vecchia approximation. Section \ref{sec:asymp} discusses asymptotic properties of Vecchia approximation for GPs. In Section \ref{sec:numeric}, we illustrate the performance of maximum Vecchia-likelihood (MVL) estimation and GP prediction based on Vecchia approximation, and compare them to those of alternative approximation approaches. Section \ref{sec:conc} concludes and discusses future work. 
The Supplementary Material contains proofs and additional numerical results.


\section{Gaussian process (GP) and Vecchia approximation \label{sec:vecchia}}

Let $\{y(\bx): \bx \in \inputdomain\}$ be the process of interest on a domain $\inputdomain \subset \mathbb{R}^d$, $d \in \mathbb{Z}^+$. In spatial applications we often have $d=2$ dimensions. Assume that $y(\cdot) \sim \GP(\mu,K)$ is a Gaussian process with mean function $\mu$ and covariance function $K = K_{\bftheta}$. Let $\bftheta \in \bm{\Theta}$ be a vector of covariance parameters, where $\bfTheta$ is the $q$-dimensional parameter space. For the sake of simplicity, we assume $\mu = 0$, as assumed in most of the literature. The covariance function $K$ must be symmetric and positive definite, that is, the covariance matrix $\bK_{n} = \bK_{n} (\bftheta)$ induced by $K$ evaluated at every pair of $n$ inputs $\locs_n = \{ \bx_1, \ldots, \bx_n \}$ through ${(\bK_{n})}_{ij} = K(\bx_i , \bx_j)$, $1 \le i, j \le n$, is symmetric and positive definite as well. For convenience, we assume that the inputs $\bx_1 , \ldots , \bx_n$ are distinct for any $n$. A major computational issue with GP inference is that direct use of the GP likelihood requires Cholesky-factorizing the data covariance matrix $\bK_n$ at a cost that is cubic in $n$. This is prohibitively expensive for datasets exceeding $\approx 10^5$ entries.

Vecchia approximation of GPs is motivated by the ordered conditional representation of the joint density function: Let $\by = (y_1, \ldots, y_n)^{\top}$ be the vector of GP realizations $y_i = y(\bx_i)$ at $\locs_n$. The density function $\dens_n(\by ; \bftheta)$ can be expressed as a product of univariate conditional densities, that is,
$
\dens_n(\by ; \bftheta) = \prod_{i=1}^n \dens_{\bftheta}(y_i|\by_{h(i)}),
$
where $h(i) = (1, \ldots, i-1)$ is the set of ``past'' indices. Vecchia approximation is given by
\begin{equation}
\label{eq:vecchia}
\hat{\dens}_{n,m}(\by ; \bftheta) = \prod_{i=1}^n \dens_{\bftheta}(y_i|\by_{g(i)}),
\end{equation}
where $\by_{h(i)}$ is replaced by $\by_{g(i)}$, $g(i) \subset h(i)$ is the conditioning index set of size $|g(i)| \leq m$. The (maximum) size of conditioning sets, $m$, is a complexity parameter that trades off approximation accuracy for computational cost in the sense that  Kullback-Leibler (KL) divergence between $\dens_n(\by)$ and $\hat{\dens}_{n,m} (\by)$ decreases as $m$ increases \citep{Guinness2016a,Schafer2020}; in the extreme cases, Vecchia approximation assumes independent observations if $m = 0$ and is exact if $m = n - 1$, but the computational feasibility of Vecchia approximation relies on small $m \ll n$.

The performance of Vecchia approximation depends on the choice of ordering of the observations 
in $\by$, and also the choice of conditioning sets $g(2) , \ldots , g(n)$.
In particular, \citet{Guinness2016a} indicated that combining maximum-minimum-distance (maximin) ordering with nearest-neighbor conditioning outperforms alternative approaches.
\citet{Schafer2017,Schafer2020} derive the maximin ordering as a highly simplified multiresolution basis and rigorously prove the resulting exponential decay of conditional correlations.
Hence, we here consider maximin ordering, which orders each successive input to maximize the minimum distance to all previously ordered inputs, 
and nearest-neighbor conditioning whose the $i$th conditioning set consists of the $\min (m, i-1)$ nearest inputs ordered previously. For the sake of simplicity, we assume henceforth that $\by$ follows a maximin ordering.

Vecchia approximation can be highly accurate even if $m$ is small (e.g., 10 to 50), for many popular covariance functions including Mat\'ern covariance functions. This can be explained by the screening effect for popular covariance functions \citep[e.g.,][]{Stein2002,Stein2011,Porcu2020}. Figure \ref{fig:corcond} illustrates the screening effect exploited by Vecchia approximation for Mat\'ern covariance structure with regular locations. Although the magnitude of the screening effect heavily depends on covariance parameters (e.g., range and smoothness), it appears in the figure that conditional correlation maps induced by conditioning on all past observations (top row) are similar to those induced by conditioning on nearest-neighbor past observations (bottom row). Also, it appears that the the screening effect strengthens for increasing $n$.

\begin{figure}[htbp]
    \centering
    \includegraphics[width = 0.2\linewidth]{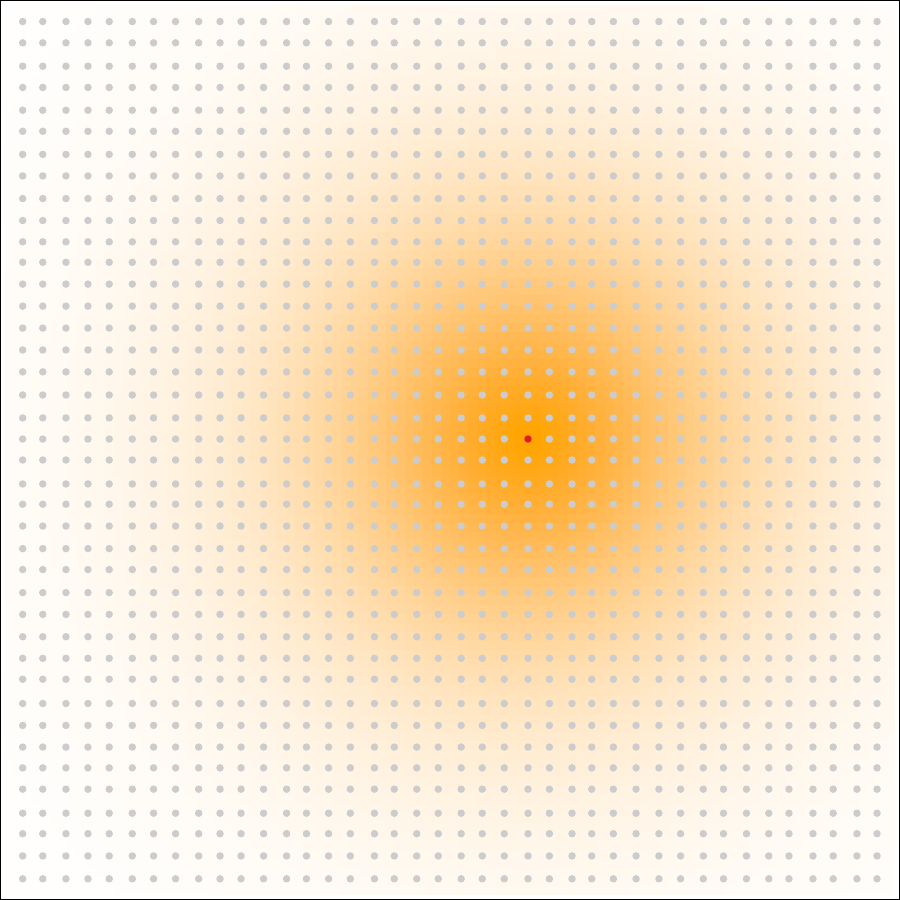}%
    \includegraphics[width = 0.2\linewidth]{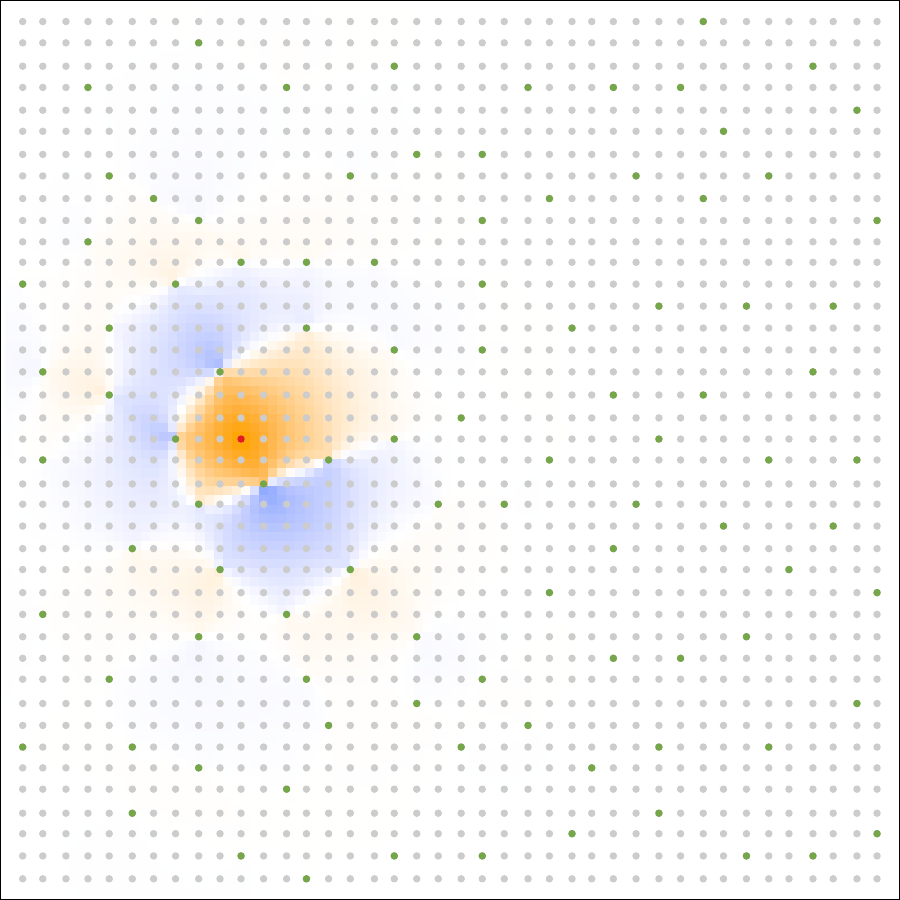}%
    \includegraphics[width = 0.2\linewidth]{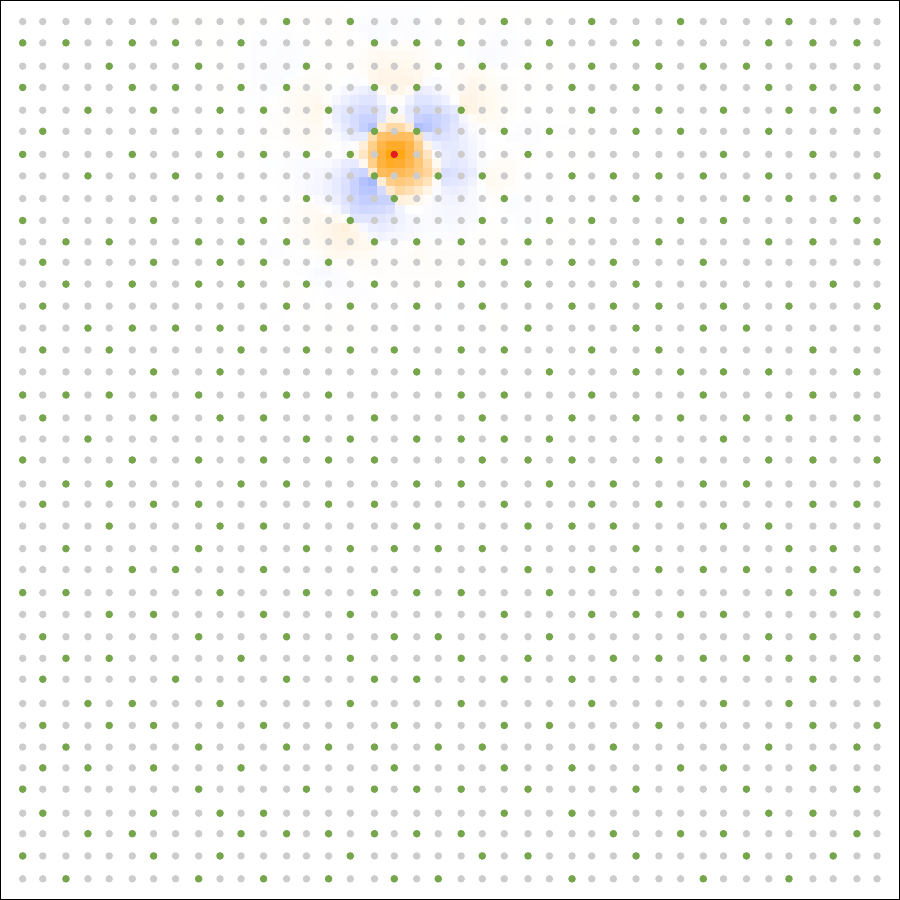}%
    \includegraphics[width = 0.2\linewidth]{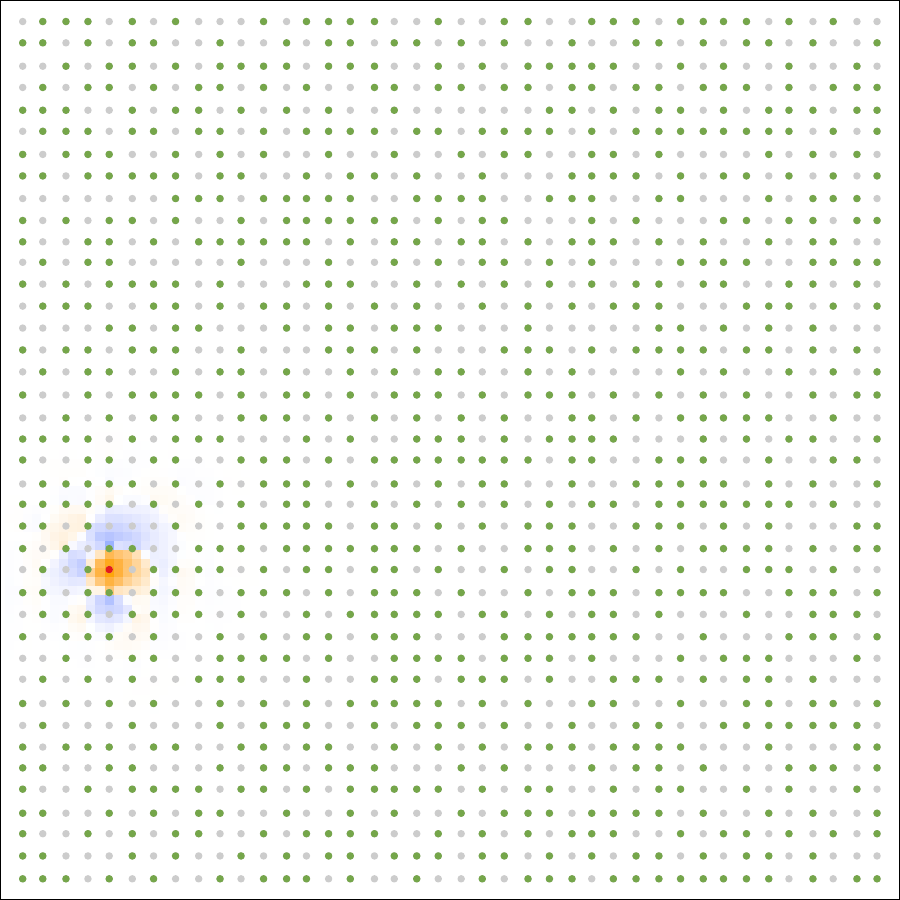}%
    \includegraphics[width = 0.2\linewidth]{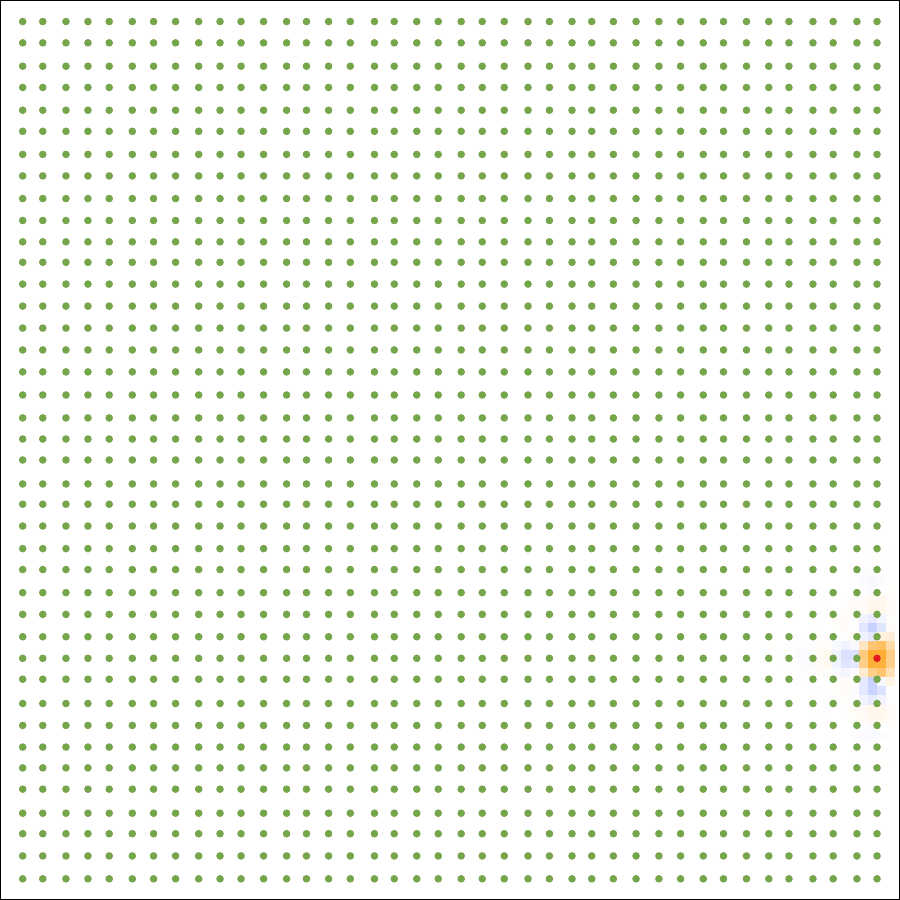}%
    \vspace{0.5em}
    \includegraphics[width = 0.2\linewidth]{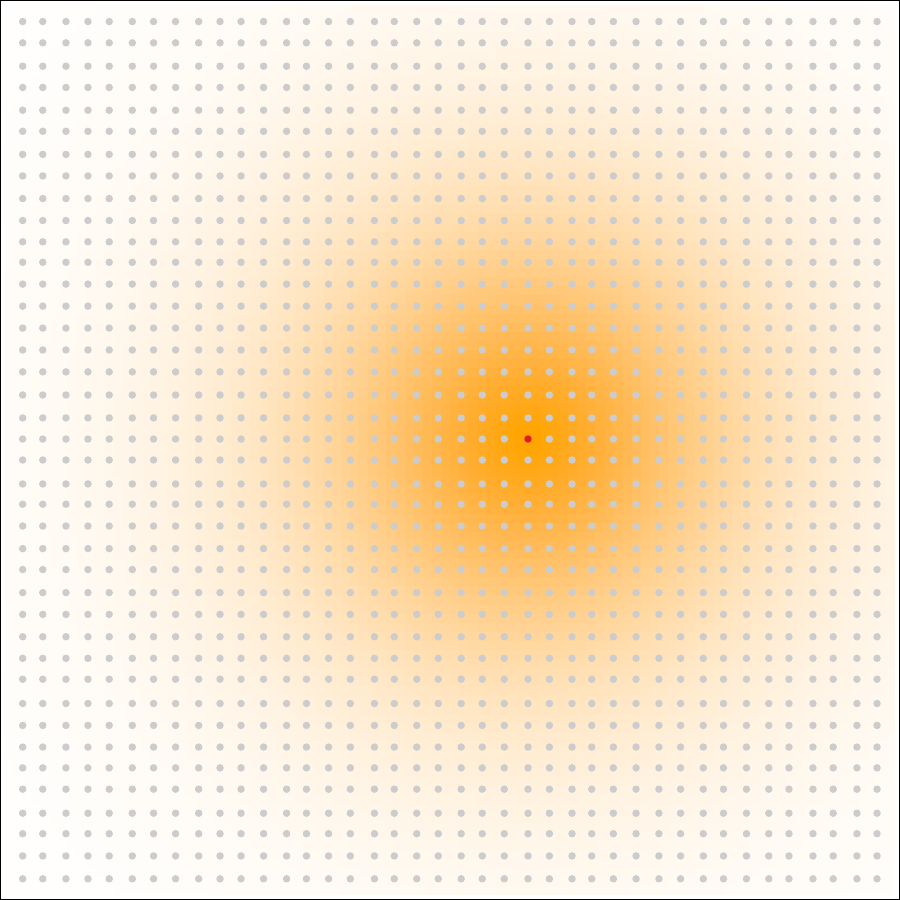}%
    \includegraphics[width = 0.2\linewidth]{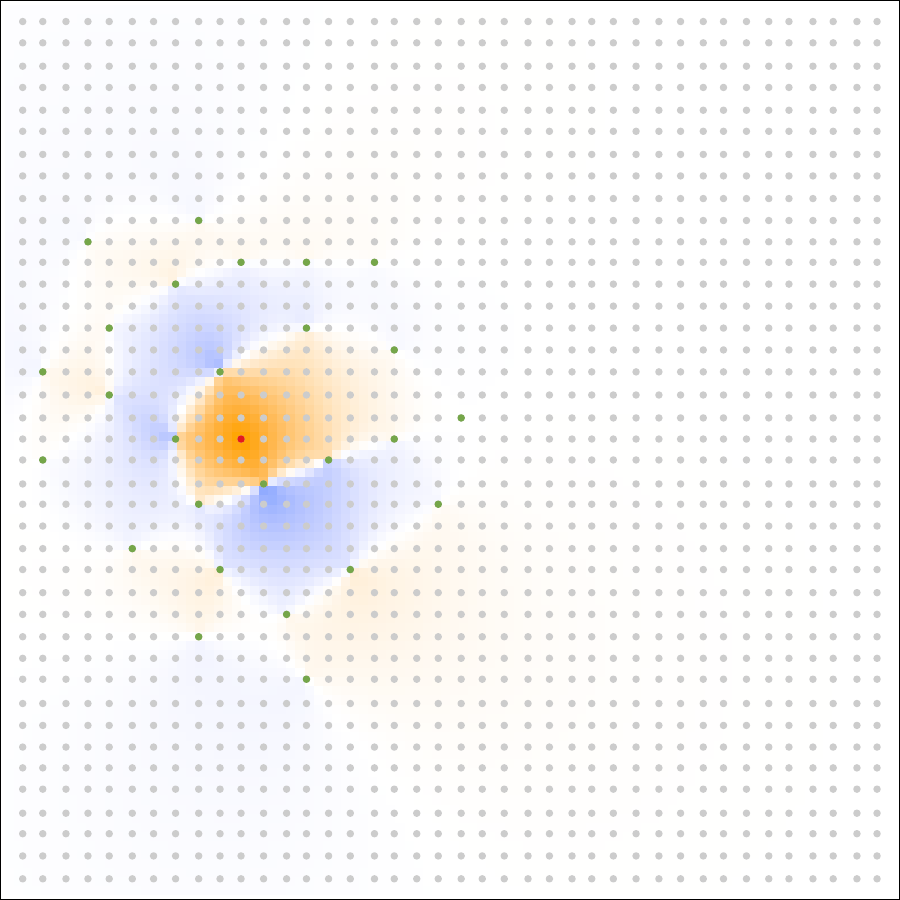}%
    \includegraphics[width = 0.2\linewidth]{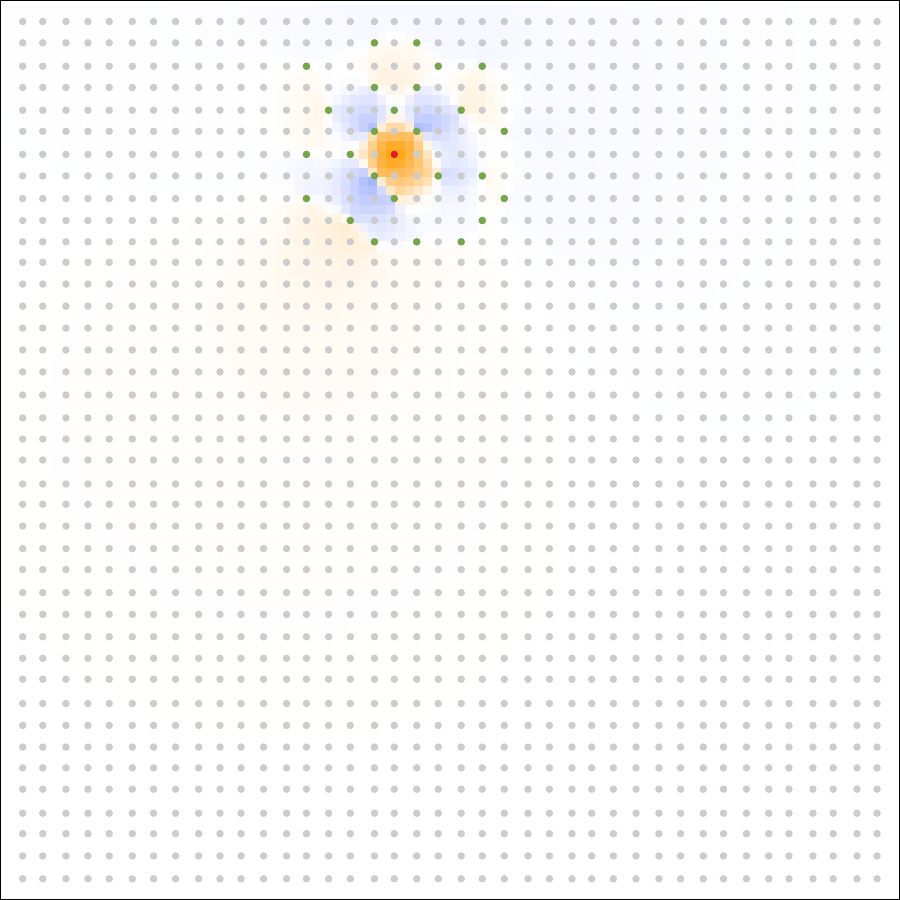}%
    \includegraphics[width = 0.2\linewidth]{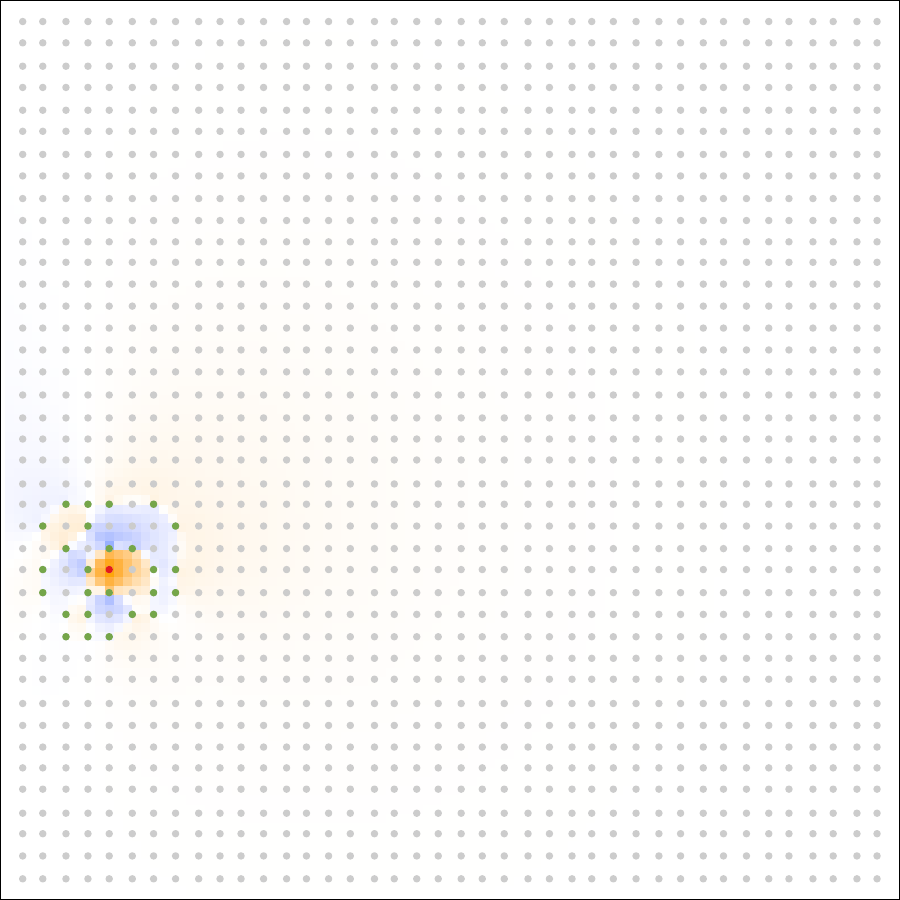}%
    \includegraphics[width = 0.2\linewidth]{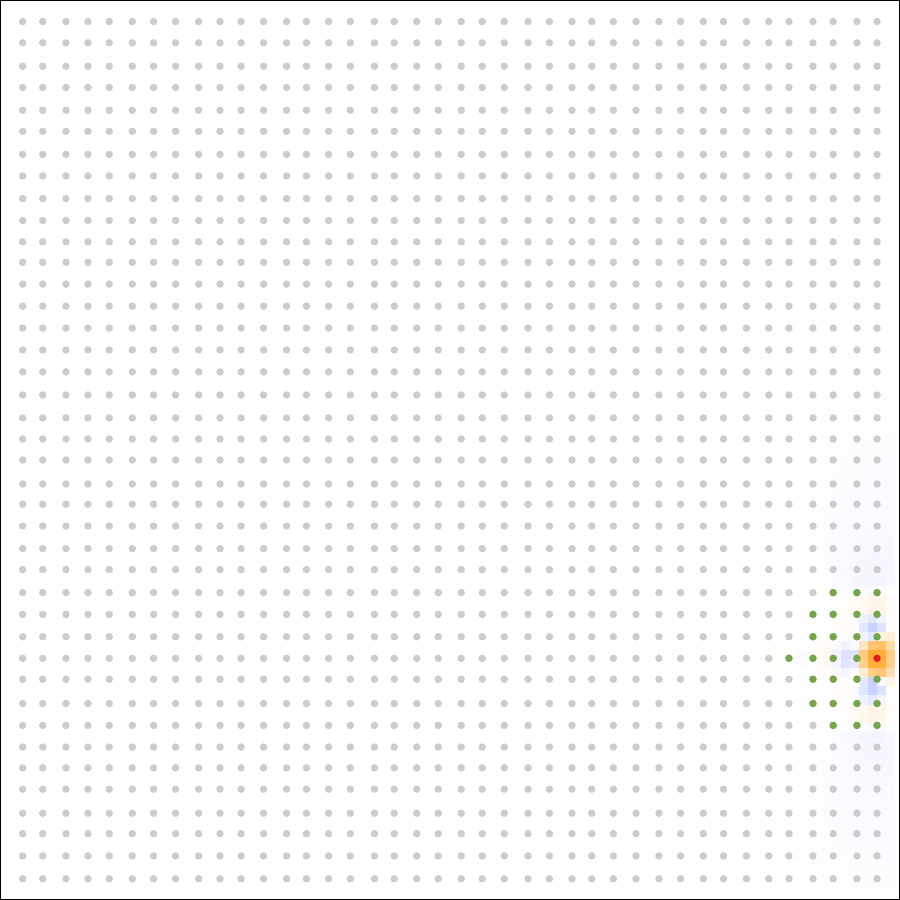}%
    
    \caption{For $i = 1, 100, 400, 900, 1600$ (from left to right), positive (orange) and negative (blue) conditional correlations with the $i$th input $\bx_i$ in the maximin ordering (red points) conditional on all (top) and $m=26$ nearest (bottom) previously ordered inputs (green points), for a GP with Mat\'ern covariance (range $r = 0.1$ and smoothness $\nu = 2$) on a grid of size $n= 40 \times 40 = 1600$. The conditional correlations given only nearest neighbors are almost identical to corresponding conditional correlations given all past observations, which implies that the two likelihoods corresponding to the conditional correlation maps are also similar to each other. This figure is inspired by Figure 5 in \citet{Schafer2020}.
    }
    \label{fig:corcond}
\end{figure}

Vecchia approximation not only leads to computational advantages but also exhibits useful inferential properties. First of all, Vecchia approximation defines a proper Gaussian likelihood so that it is possible to apply any likelihood-based inference method. In particular, the Vecchia likelihood is the Gaussian likelihood with zero-mean and covariance 
$$\hat{\bK}_{n, m} = \hat{\bK}_{n, m} (\bftheta) = (\bU_{n, m} (\bftheta) \bU_{n, m} (\bftheta)^\top)^{-1} = (\bU_{n, m} \bU_{n, m}^\top)^{-1}$$
where $\bU_{n, m}$ is a sparse upper-triangular Cholesky factor of the implied precision matrix of the data vector $\by$ \citep[Proposition 1 in][]{Katzfuss2017a}. Each of the $n$ columns of $\bU_{n, m}$ can be computed in $\mathcal{O}(m^3)$ time, completely in parallel. For a single-core implementation, the computational cost of the Vecchia approximation is $\order(nm^3)$, although this can be reduced to $\order(nm^2)$ using a grouping strategy \citep{Stein2004, Ferronato2015, Guinness2016a, Schafer2020}. 
Regarding the accuracy of Vecchia approximation, $\bU_{n, m}$ is \emph{optimal} in terms of KL divergence between $\normal_n(\mathbf{0}, \bK_n)$ and $\normal_n(\mathbf{0},(\bU_{n, m} \bU_{n, m}^\top)^{-1})$ for a given ordering of the observations and sparsity pattern imposed by the conditioning sets \citep{Schafer2020}.


\section{Asymptotic properties of Vecchia approximation \label{sec:asymp}}

We are interested in establishing consistency and asymptotic normality of maximum Vecchia-likelihood (MVL) estimators based on those of maximum exact-likelihood (ML) estimators, and also exploring asymptotic behavior of GP prediction based on Vecchia approximation.

This section is divided into two parts: The first part considers GPs with covariance functions associated with particular types of spectral densities and boundary conditions. These assumptions are an idealization of a large class of popular covariance functions \citep[e.g.,][]{Schafer2017,Schafer2020}. For regularly (or roughly homogeneously) distributed data points in the domain, the covariance functions identified as Green’s functions of integer-order elliptic partial differential equations are asymptotically smooth enough to hold the assumptions. A notable example is the Mat\'ern class with half-integer smoothness and zero Dirichlet boundary condition. Regarding this, GPs with such covariance functions are referred to as GPs with boundary conditioning. The second part considers GPs with Mat\'ern covariance functions but without boundary conditioning, which are popular for spatial data. 

\subsection{Vecchia approximation to GPs with boundary conditioning \label{sec:contGP}}

In this section, the following assumptions are made:
\begin{itemize}
    \item[\textbf{(A1)}] Regular domain: $\inputdomain$ is a bounded domain of $\mathbb{R}^d$ with Lipschitz boundary $\boundary = \partial \inputdomain$, that is, $\boundary$ can be locally represented by a graph of a Lipschitz continuous function.
    \item[\textbf{(A2)}] Regular locations: Let $\left\{ \locs_n \right\}_{n \in \mathbb{Z}^+}$ be an increasing sequence of ordered finite subsets $\locs_n$ of $\inputdomain \setminus \boundary$ satisfying that there exist $c_1 , c_2 >0$ such that $$0 < c_1 < \delta_n < c_2 < \infty,$$ where $\delta_n$ is a measure of homogeneity of $\locs_n$ defined by 
    \begin{equation}
        \label{eq:delta}
    \delta_n = \frac{\min_{\bx_i \neq \bx_j \in \locs_n} \text{dist} \left( \bx_i ,\, \lbrace \bx_j \rbrace \cup \boundary \right)}{\max_{\bx \in \inputdomain} \text{dist} \left( \bx ,\, \locs_n \cup \boundary \right)},
    \end{equation}
    and $\text{dist} (\bx , \locs) = \inf_{\bx' \in \locs} \| \bx - \bx' \|.$ 
    \item[\textbf{(A3)}] Reciprocal spectral density: The spectral density $\psi^{-1}$ of the covariance function $K_0$ is the reciprocal of an even polynomial, that is, $\psi(\bfomega) = \sum_{|j| \leq s} a_{j} \bfomega^{2 j}$ with $s > d$, and is uniformly elliptic in the sense that there exists a constant $c_{\psi} > 0$ such that for all $\bfomega \in \mathbb{R}^d$,
    \begin{equation}
        \psi(\bfomega) \geq c_{\psi} \left(1 + \|\bfomega\|^{2s}\right).
    \end{equation}
    \item[\textbf{(A4)}] Boundary conditioning: The covariance function $K$ of the GP is defined as $K_0$ conditional on the event that $y(\bx) = 0$ for all $\bx \in \mathbb{R}^d \setminus \inputdomain$.
\end{itemize}
The assumption (A1) is mild and can be considered valid if an input domain is either a $d$-dimensional (hyper)rectangle or a (hyper)sphere. (A2) is a standard regularity condition on the inputs \citep{Johnson1990}, defined as the ratio between the smallest interpoint distance in $\locs_n$, and the distance to the most remote point in $\inputdomain$ relative to the points in $\locs_n$. This ensures that the points in $\locs_n$ cover $\inputdomain$ with out any two points in $\locs_n$ getting too close together, mimicking a maximin ordering. For instance, consider inputs $\locs_{n} = \{ \bs_{ij} = (i/\sqrt{n+1} , j/\sqrt{n+1}) ; i, j = 1, \ldots , \sqrt{n} \}$ in the unit-square input domain $\inputdomain = [0,1]^2$ for a square number $n$. Then, the homogeneity measure $\delta_n$ is $\sqrt{2}$, since the numerator in \eqref{eq:delta} is $1/\sqrt{n+1}$ and the denominator is $1 / (\sqrt{2(n+1)})$. (A3) is broad enough to include a wide range of covariance functions, including a Mat\'ern class with smoothness $\nu$ satisfying $2(\nu+d/2) \in \mathbb{Z}^+$. \cite{Schafer2020} suggested that even the Mat\'ern class with non-half-integer smoothness and Cauchy class \citep{Gneiting2004} are subject to similar behavior. Although it can be seen difficult in many applications to verify the assumption (A4), there are several applications whose boundary assumptions share similarities with our boundary conditioning assumption \citep[e.g.,][]{Solin2019,Swiler2020}. 

To describe the exact or lower bound on the size of conditioning sets, $m = m(n)$, in terms of $n$, we use the standard $\knuth$-notation \citep{Knuth1976}: By definition, $m(n) = \knuth ( f(n) )$ is equivalent to $f(n) = \order (m (n))$, where $f$ is a function of $n$. For the sake of simplicity, we suppress the argument $n$ for $m$.

\begin{proposition}[Convergence in KL divergence of Vecchia approximation] \label{prop:klconv}
    Assume (A1)--(A4). For any $\bftheta \in \bfTheta$,
    \begin{equation}
        m = \knuth(\log^d n) \quad \Rightarrow \quad \KL \Big( \dens_{n} (\cdot ; \bftheta) \| \hat{\dens}_{n,m} (\cdot ; \bftheta) \Big) \rightarrow 0 \, \text{ as } n \rightarrow \infty,
    \end{equation}
    where the KL divergence is defined between Gaussian measures with the exact likelihood $\dens_{n}$ and Vecchia likelihood $\hat{\dens}_{n,m}$ under (A4). Note that the KL divergence decays faster than any fixed-order polynomial as a function of $n$. 
\end{proposition} 

Proposition \ref{prop:klconv} implies the convergence in KL divergence of the Vecchia-GP distribution to the exact-GP distribution, given that $m$ grows at least in the order of $\log^d (n)$. Notably, this provides a theoretical criterion for the selection of the size of conditioning sets, $m$. The constant in the $\knuth$ notation for the lower bound on $m$ depends on the input dimension $d$, input domain $\inputdomain$, homogeneity measure $\{ \delta_n \}$, and covariance function $K_{\bftheta}$. Although these relationships are not as simple as the one obtained for $n$, we believe that the constant is relatively small; this conjecture is explored numerically in Section \ref{sec:numeric}.

\begin{proposition}[Asymptotic equivalence of Vecchia GP prediction] \label{prop:predbound}
    Assume (A1), (A3), and (A4). For any $\bftheta \in \bfTheta$ and new input $\bx^* \in \inputdomain \setminus \boundary$, if (A2) holds with $\locs^*_n = \locs_n \cup \{ \bx^* \}$,
    \begin{equation}
        m = \knuth(\log^d n) \; \Rightarrow \; \KL \Big( \dens_{\bftheta} (y^* | \by_{1:n}) \,\Big\|\, \dens_{\bftheta} (y^* | \by_{g^*}) \Big) \rightarrow 0 \; \text{ as } \; n \rightarrow \infty,
    \end{equation}
    where $y^* = y(\bx^*)$, $g^* \subset \{ 1, \ldots , n \}$ is the nearest-neighbor conditioning index set of size $m$ for $\bx^*$, and the KL divergence is defined between Gaussian measures with densities $\dens_{\bftheta} (y^* | \by_{1:n})$ and $\dens_{\bftheta} (y^* | \by_{g^*})$ under (A4).
\end{proposition}

Proposition \ref{prop:predbound} implies that the true predictive distribution of unobserved variable at new input is asymptotically equivalent with its Vecchia approximation, which can be obtained in $\order (\log^{3d} n)$ time, given that $m$ grows polylogarithmically with $n$. This is a useful new result on the screening effect, indicating that only the nearest $\log^d(n)$ variables are necessary for prediction in our setting. The proposition can be shown by Lemma \ref{lemm:regbound} in the Supplementary Material and ordering the new input last. 

Let $\bftheta_0 \in \bfTheta$ be the data-generating covariance parameter. To discuss consistency of ML and MVL estimators for $\bftheta$, we need more assumptions:
\begin{itemize}
    \item[\textbf{(A5)}] Compact parameter space: $\bftheta$ is an element of a closed and bounded subset $\bfTheta$ of $\mathbb{R}^q$, $q \in \mathbb{Z}^+$.
    \item[\textbf{(A6)}] Smooth log-likelihoods: For any $n,m \in \mathbb{Z}^+$, $\log \dens_n(\by ; \bftheta)$ and $\log \hat{\dens}_{n,m}(\by ; \bftheta)$ are twice continuously differentiable on the interior of $\bfTheta$ almost surely under $P_{\bftheta_0}$.
    \item[\textbf{(A7)}] Uniformly bounded first derivatives: There exists a constant $c_\nabla > 0$ such that 
    $$\mathbb{E}_{\bftheta_0} \Big[ \sup_{\bftheta \in \bfTheta} \| \frac{\partial}{\partial \bftheta} \log \dens_n(\by ; \bftheta) \| \Big] + \mathbb{E}_{\bftheta_0} \Big[ \sup_{\bftheta \in \bfTheta} \| \frac{\partial}{\partial \bftheta} \hat{\dens}_{n,m}(\by ; \bftheta) \| \Big] < c_\nabla,$$
    for all $n \in \mathbb{Z}^+$.
    \item[\textbf{(A8)}] Identifiability: $P_{\bftheta_1} \equiv P_{\bftheta_2}$ implies $\bftheta_1 = \bftheta_2$ for any $\bftheta_1 , \bftheta_2 \in \bfTheta$, where $P_{\bftheta}$ is a Gaussian measure with respect to the GP with covariance parameter $\bftheta \in \bfTheta$.
\end{itemize}
The above assumptions are fairly standard in the literature. (A5) - (A7) impose mild conditions on covariance functions and are reasonable in many applications, especially for evenly distributed inputs in the domain. (A8) is on equivalence of Gaussian measures induced by GPs and needs careful consideration.

Under assumptions (A5)--(A8), we can define the ML and MVL estimators, respectively, for the covariance parameter $\bftheta \in \bfTheta$ as follows:
\begin{equation}
    \hat{\bftheta}_{n} = \argmax_{\bftheta \in \bfTheta} \dens_n(\by ; \bftheta) 
    \quad \text{and} \quad 
    \hat{\bftheta}_{n,m} = \argmax_{\bftheta \in \bfTheta} \adens_{n,m}(\by ; \bftheta),
\end{equation}
for $n, m \in \mathbb{Z}^+$. Note that fixed-domain asymptotic properties of the ML estimator $\hat{\bftheta}_{n}$ have been studied over decades, as mentioned in Section \ref{sec:intro}. For instance, \cite{Zhang2004} provided substantive results on identifiability for parameters of Mat\'ern Gaussian processes. \cite{Anderes2010} and \cite{Kaufman2013} studied consistencies for identifiable parameters under low- and high-dimensional fixed-domain asymptotics, respectively.

\begin{proposition}[Asymptotic equivalence of MVL estimation] \label{prop:mvlebound}
    Let (A1)--(A8) hold. Assume that $\hat{\bftheta}_n$ is consistent and asymptotically normally distributed, that is, $$\hat{\bftheta}_n \rightarrow \bftheta_0 \, \text{ in probability under } P_{\bftheta_0}$$ and there exists a sequence $\{ \bA_n (\bftheta_0) \}$ of non-random and positive-definite matrices such that $$\bA_n (\bftheta_0) (\hat{\bftheta}_n - \bftheta_0) \rightarrow \normal_q (\bfzero , \bI_q) \, \text{ in distribution under } P_{\bftheta_0}.$$ Then, $\hat{\bftheta}_{n,m}$ is consistent and asymptotically normally distributed, that is, there exists a sequence $m = \knuth(\log^d n)$ such that $$\hat{\bftheta}_{n,m} \rightarrow \bftheta_0 \, \text{ in probability under } P_{\bftheta_0}$$ and $$\bA_n (\bftheta_0) (\hat{\bftheta}_{n,m} - \bftheta_0) \rightarrow \normal_q (\bfzero , \bI_q) \, \text{ in distribution under } P_{\bftheta_0}.$$ 
\end{proposition}

Proposition \ref{prop:mvlebound} implies consistency and asymptotic normality of the ML estimator $\hat{\bftheta}_{n}$ based on exact GP inference imply those of the MVL estimator $\hat{\bftheta}_{n,m}$ based on Vecchia GP inference, given that $m$ grows at least polylogarithmically with $n$. Furthermore, it implies that the efficiency of the MVL estimators is identical to that of ML estimators. This means that asymptotically exact GP inference can be obtained in $\order (n \log^{3d} n)$ time using Vecchia approximation. As mentioned earlier, this complexity can be improved to $\order (n \log^{2d} n)$ time, for example by using the grouping strategy of \citet{Schafer2020}.  Our statement in Proposition \ref{prop:mvlebound} is inspired by Corollary 2 of \cite{Pedersen1995}.

\subsection{Vecchia approximation to GPs without boundary conditioning \label{sec:uncontGP}}

In various applications, including geostatistics and computer-model emulation, the validity of the boundary conditioning assumption is open to question. Although the theoretical results presented in Section \ref{sec:contGP} hold only for GPs with boundary conditioning and we do not have a rigorous analysis of asymptotic theory for GPs without boundary conditioning, we conjecture that Section \ref{sec:contGP} provides guidelines even for the use of Vecchia approximation to GPs without boundary conditioning. 

For GP prediction without boundary conditioning, we conjecture that the true predictive distribution of unobserved variable at new input can be inferred accurately by using its Vecchia approximation for sufficiently large $m$ and $n$, if the new input is not close to the boundary of the domain. This is motivated by the observation that conditional correlation maps obtained by the exact GP were nearly identical with those by its nearest-neighbor-based Vecchia approximation in Figures \ref{fig:corcond}, \ref{fig:corcond_n400}, and \ref{fig:corcond_n3600}, especially for larger $n$. In other words, because the conditional correlation with an interior point is close to zero beyond its nearest-neighbor conditioning set, removing the boundary conditioning should have a negligible effect. The conjecture is explored numerically in Section \ref{sec:numeric}.

A promising approach to investigate asymptotic properties of Vecchia-based parameter estimation without boundary conditioning is based on the theory of estimating equations, as suggested by \cite{Stein2004}. Note that \cite{Shaby2010} used a similar approach to study asymptotic properties of two-taper approximation. An estimating function for Vecchia approximation is the gradient of the log Vecchia-likelihood function $\log \adens_{n,m}$ with respect to $\bftheta$. Then, the robust information measure (or Godambe information) is defined by
\begin{equation}
    \mathcal{E}_{n,m} (\bftheta) = \mathbb{E}_{\bftheta} \Big[ \frac{\partial^2}{\partial \bftheta^2} \log \hat{p}_{n,m} \Big]^{\top} \mathbb{E}_{\bftheta} \Big[ \Big( \frac{\partial}{\partial \bftheta} \log \hat{p}_{n,m} \Big)^2 \Big]^{-1} \mathbb{E}_{\bftheta} \Big[ \frac{\partial^2}{\partial \bftheta^2} \log \hat{p}_{n,m} \Big].
\end{equation}
where $\left( \frac{\partial}{\partial \bftheta} \log \hat{p}_{n,m} \right)^2 = \left(\frac{\partial}{\partial \bftheta} \log \hat{p}_{n,m} \right) \left( \frac{\partial}{\partial \bftheta} \log \hat{p}_{n,m} \right)^{\top}$. For instance, the robust information measure is the Fisher information of the exact log-likelihood $\log \dens_n$ if $m = n-1$. See \cite{Heyde1997,Hall2014} for more details.

\begin{proposition}[Unbiased Vecchia-estimating function] \label{prop:unbias}
    The Vecchia-estimating equations are unbiased for any ordering and conditioning, that is, $$\mathbb{E}_{\bftheta} \Big[ \frac{\partial}{\partial \bftheta} \log \adens_{n,m} (\cdot , \bftheta)\Big] = \bfzero \, \text{ for any } \, \bftheta \in \bfTheta \setminus \partial \bfTheta.$$
\end{proposition}

Note that Proposition \ref{prop:unbias} and its proof are not new but are motivated by \cite{Stein2004}. Since the estimating equations for Vecchia GP inference are unbiased, the robust information measure can be used to obtain the asymptotic sampling distribution of the MVL estimator $\hat{\bftheta}_{n,m}$: Under some regularity conditions, for sufficiently large $m$ and $n$, we conjecture that
$$\hat{\bftheta}_{n,m} \sim \normal_q (\bftheta_0 , \mathcal{E}_{n,m}^{-1} (\bftheta_0)) \quad \text{approximately}.$$ 
To provide further intuition, note that we can write the Vecchia loglikelihood as
\begin{align}
\log \hat{\dens}_{n,m}(\by ; \bftheta) &= \sum_{i=1}^n \log \dens_{\bftheta}(y_i|\by_{g(i)}) \\
&= \sum_{i \in \mathcal{I}} \log \dens_{\bftheta}(y_i|\by_{g(i)}) + \sum_{i \in \mathcal{B}} \log \dens_{\bftheta}(y_i|\by_{g(i)}),
\end{align}
where $\mathcal{I}$ and $\mathcal{B}$ are sets of indices corresponding to interior points and those close to the boundary, respectively. Most of the summands (and their derivatives with respect to $\bftheta$) correspond to GP predictions for interior observations, for which removing the boundary conditioning should have little effect, as argued above, resulting in $\dens_{\bftheta}(y_i|\by_{g(i)}) \approx \dens_{\bftheta}(y_i|\by_{h(i)})$. For boundary points, the summands still result in unbiased estimating equations (see Proposition 4).
Our numerical experiments in Section \ref{sec:numeric} suggest that the MVL estimators of the covariance parameters can be very accurate for $m$ growing only polylogarithmically with $n$, even without boundary conditioning.


\section{Numerical experiments \label{sec:numeric}}

We conducted simulation experiments to illustrate fixed-domain asymptotic properties of Vecchia approximation. We considered zero-mean GPs with Mat\'ern covariance function \citep{Stein1999} 
\begin{equation}
    K(\bx , \bx') = \sigma^2 \frac{\pi^{1/2} r^{2\nu}}{2^{\nu-1} \Gamma(\nu + 1/2)} \left( \frac{\| \bx - \bx' \|}{r} \right)^{\nu} B_{\nu} \left( \frac{\| \bx - \bx' \|}{r} \right),
\end{equation}
where $\Gamma$ is the gamma function, $B_\nu$ is the modified Bessel function of the second kind, with variance $\sigma^2 = 1$, range $r = 0.1$, and smoothness $\nu = 0.5$, on the unit-square domain $\inputdomain = [0, 1]^2$. The inputs were obtained on regular grids. Note that we did not condition the GPs on the boundary for our numerical experiments.

We compared the following GP approximations: Vecchia approximation, one-taper approximation with decreasing or fixed taper range, two-taper approximation with decreasing or fixed taper range, and reduced-rank approximation with inducing points. We also compared to the exact GP if possible. The size of conditioning sets, $m$, was set to be the maximum integer smaller than $1.2^d \log^d_{10} (n)$. For the covariance tapering approaches, we used the compactly supported polynomials proposed by \cite{Wendland1995,Wendland1998} and the fixed taper range was set to be equal to the range $r = 0.1$ of the Mat\'ern covariance function. The values of decreasing taper range were chosen in order to compare taper and Vecchia approximations at similar computational complexity parameter values, while their computing times were not used for comparison. Vecchia approximation has quasilinear time complexities in $n$ with our experiment setup, while fixed-range covariance tapering has the same order of computational complexity as the exact GP. Details on asymptotic properties of maximum one-tapered-likelihood estimators can be found in \citet{Du2009} and \citet{Wang2011} and details of maximum two-tapered-likelihood estimators in \cite{Kaufman2008} and \cite{Shaby2010}. For the reduced-rank approach, the number of inducing points were set to be approximately equal to the average size of conditioning sets of Vecchia approximation, and the inducing points were chosen as the first $m$ ordered inputs in the maximin ordering, since they formed a domain-covering set \citep{Guinness2016a}. 

Figure \ref{fig:comparison_prediction} shows predictive performances of taper, reduced-rank, and Vecchia approximations at the center of the domain. Exact-GP prediction results are also plotted for the baseline. For the tapering approaches, since GP prediction does not directly use sample covariance matrices (but uses model covariance matrices), predictions from one-taper approximation are identical with predictions from two-taper approximation. While predictions from covariance tapering with decreasing taper range performed reasonably well in terms of mean square prediction error (MSPE) and variance of predictive distribution, they did not improve for increasing $n$ in terms of KL divergence to predictive distribution based on exact GP. Predictions from tapering with fixed taper range performed better, but their computational cost can be prohibitive for very large $n$, according to the top right panel. We found that predictions from reduced-rank approximation were unstable and also became less accurate for larger $n$. Predictions from Vecchia approximation outperformed those from the competing methods and they were almost identically accurate with those based on exact GP. Furthermore, the KL divergence between predictive distributions at $\bx_{n+1}$ based on Vecchia and exact GPs was decreasing with $n$, which supports our theoretical finding in Proposition \ref{prop:predbound} and the conjecture in Section \ref{sec:uncontGP}, implying that exact GP prediction can be approximated accurately and efficiently by Vecchia GP prediction even without boundary conditioning. 

\begin{figure}[htbp]
    \centering
    \includegraphics[width=\textwidth]{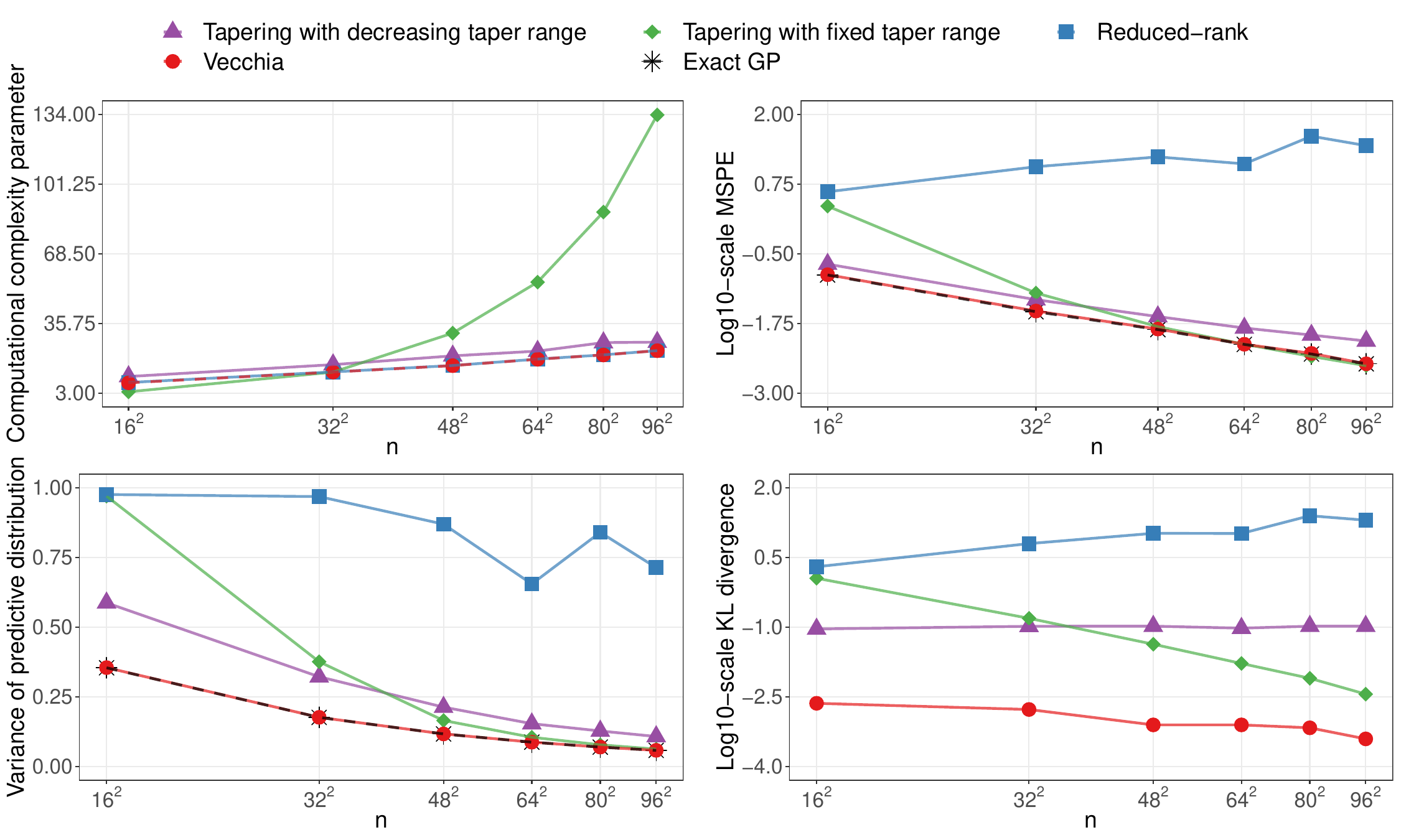}
    \caption{Comparison of predictive performances of taper, reduced-rank, and Vecchia GP approximations at $\bx_{n+1} = (0.5, 0.5) \in [0,1]^2$ for 400 synthetic datasets: The top left panel compares the average number of non-zero entries of covariance matrix per observation for taper approximations, the number of inducing points for reduced-rank approximation, and the average size of conditioning set for Vecchia approximation. The top-right and bottom-left panels present log-scale mean square prediction error (MSPE) and variance of the predictive distributions, respectively. The bottom right panel compares KL divergences between $p(y_{n+1} | \by_{1:n})$ and $\hat{p}(y_{n+1} | \by_{1:n})$ based on the different approximations. Note that all the parameters were assumed to be known and the $x$-axes of the panels are on a log scale.}
    \label{fig:comparison_prediction}
\end{figure}

Figure \ref{fig:comparison_mle} compares root mean square error (RMSE) for estimating the covariance parameters $\sigma^2$ and $r$ based on taper, reduced-rank, and Vecchia approximations. Also, as a baseline, RMSE for exact GP inference is plotted. The RMSE values for Vecchia approximation were nearly identical to those for exact GP, which implies that Vecchia approximation with polylogarithmically increasing $m$ does not impair efficiency of maximum GP-likelihood estimation, even without boundary conditioning. As pointed out by \cite{Kaufman2008}, one-taper approximation can be severely biased, resulting in very large RMSE values. The RMSE values for two-taper and reduced-rank approximations were higher than those for Vecchia approximation and exact GP; in addition, an important finding is that RMSE for two-taper approximation with decreasing range and for reduced-rank approximation did not decrease for increasing $n$. Although the two-taper approximation with fixed range resulted in decreasing RMSE for increasing $n$, this approximation is computationally intractable for very large $n$ (and has the same cubic complexity as the exact GP). For a more detailed assessment of bias and variance of the parameter estimates, see Figure \ref{fig:comparison_mle_boxplot} showing box plots of the estimates and $90\%$ intervals of their sampling distributions.

\begin{figure}[htbp]
    \centering
    \includegraphics[width=\textwidth]{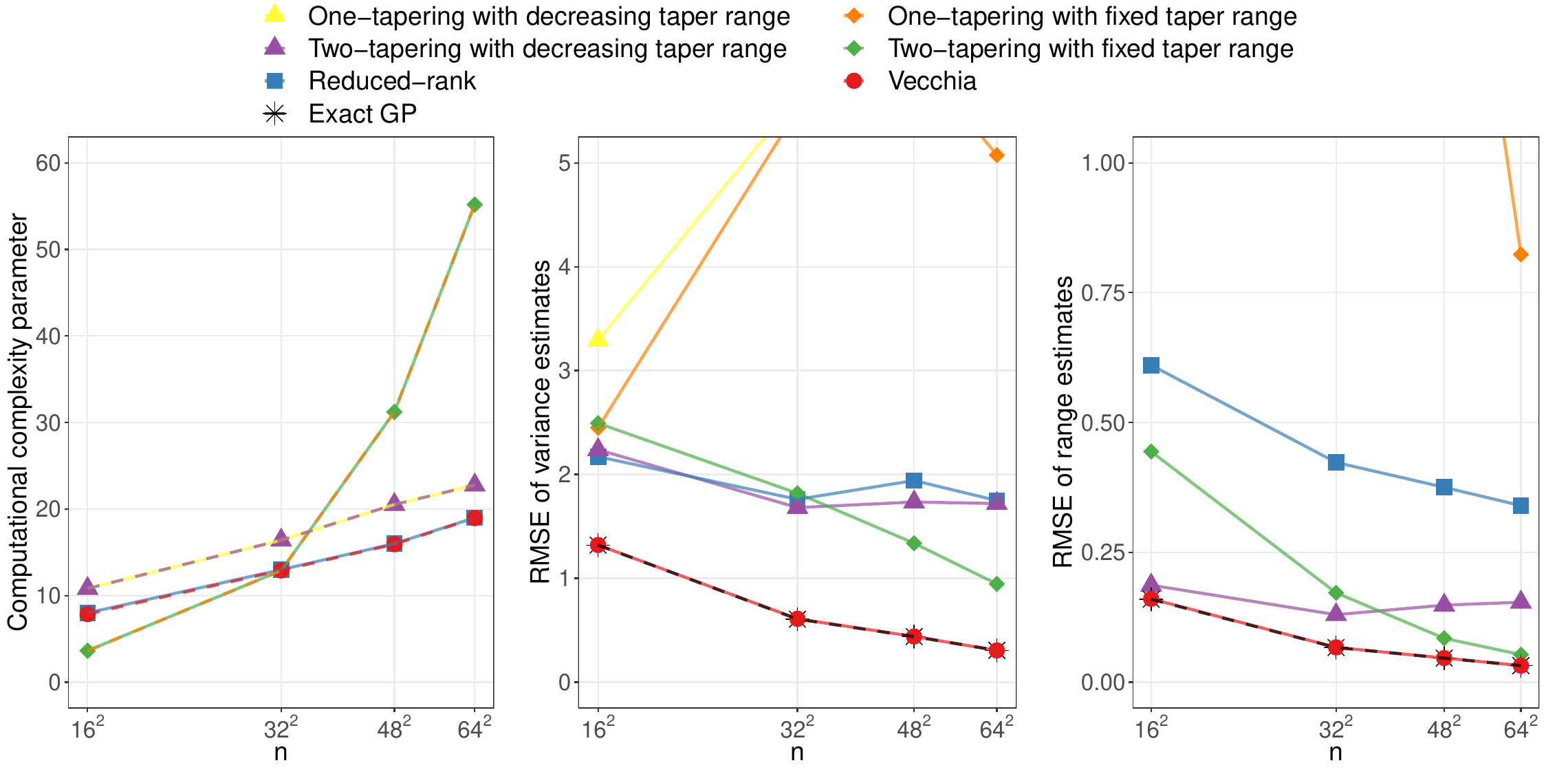}
    \caption{Comparison of maximum-approximate-likelihood estimators based on taper, reduced-rank, and Vecchia approximations: The left panel compares the average number of non-zero entries of covariance matrix per observation for taper approximations, the number of inducing points for reduced-rank approximation, and the average size of conditioning set for Vecchia approximation. The center (right) panel shows RMSE for estimating the variance (range) parameter based on exact and approximate GP likelihoods for 200 synthetic datasets. Vecchia and exact-GP RMSEs were nearly identical. While estimating a parameter, the other parameter was assumed to be known. Data were generated on a regular grid of $n$ inputs on the unit square domain from a GP with Mat\'ern covariance with variance $1$, range $0.1$, and smoothness $0.5$. Note that the $x$-axes of the panels are on a log scale.}
    \label{fig:comparison_mle}
\end{figure}


\section{Discussion \label{sec:conc}}

We have studied fixed-domain asymptotic properties of Vecchia approximation for Gaussian processes. We provided new findings suggesting the following: Given that the size of conditioning sets grows polylogarithmically with the data size, consistency and asymptotic normality of maximum exact-likelihood (ML) estimators imply those of maximum Vecchia-likelihood (MVL) estimators and exact GP prediction can be approximated accurately by Vecchia GP prediction under boundary conditioning. These findings were supported by our numerical experiments, which also showed that Vecchia approximation outperformed its alternatives such as covariance tapering and reduced-rank approximation for both GP prediction and parameter estimation. These suggest that based on the screening effect, Vecchia-based GP inference with quasilinear complexity is asymptotically equivalent to exact GP inference with cubic complexity.

While we largely focused on fixed-domain asymptotics with no nugget effect but with boundary conditioning here, we believe our findings can inspire new theoretical and experimental investigations in different settings, since our approach yields a general framework for studying asymptotic properties of likelihood approximation techniques. For instance, \cite{Mardia1984} provided interesting results concerning the consistency and asymptotic normality of ML estimators of model parameters under increasing-domain asymptotics. This can a good starting point for studying increasing-domain asymptotic properties of Vecchia approximation, as we did for fixed-domain asymptotics in this paper. Furthermore, for a GP with additive noise, one can use the Vecchia approach to approximate a latent noise-free GP, as suggested by \cite{Schafer2020} and \cite{Kang2021}. While this introduces additional computational challenges, it has been shown numerically that using the incomplete Cholesky factorization (IC) introduces only negligible additional error. This can be useful for studying asymptotic properties of Vecchia approximation for a given noise or nugget term.

\footnotesize
\appendix
\section*{Acknowledgments}

Matthias Katzfuss's research was partially supported by National Science Foundation (NSF) Grants DMS--1654083 and DMS--1953005 and by the Office of the Vice Chancellor for Research and Graduate Education at the University of Wisconsin--Madison with funding from the Wisconsin Alumni Research Foundation. Florian Sch{\"a}fer's research was supported by the Office of Naval Research under grant N00014-23-1-2545. Myeongjong Kang is currently affiliated with Merck $\&$ Co. Inc.,
North Wales, PA, USA.

\bibliographystyle{apalike}
\bibliography{ref}


\newpage
\emptythanks

\setcounter{section}{0}
\setcounter{equation}{0}
\setcounter{figure}{0}
\def\theequation{S\arabic{section}.\arabic{equation}}
\def\thesection{S\arabic{section}}

\renewcommand{\thetable}{S\arabic{table}}  
\renewcommand{\thefigure}{S\arabic{figure}}

\title{Supplementary material for ``Asymptotic properties of Vecchia approximation for Gaussian processes"}

\author{Myeongjong Kang\thanks{Department of Statistics, Texas A\&M University} \and Florian Sch\"afer\thanks{School of Computational Science and Engineering  Georgia Institute of Technology} \and Joseph Guinness\thanks{Department of Statistics and Data Science, Cornell University} \and Matthias Katzfuss\thanks{Department of Statistics, University of Wisconsin--Madison (\texttt{katzfuss@gmail.com})}}

\date{}


\maketitle


\section{Proofs \label{app:proofs}}

This section contains the proofs of technical statements. 

\begin{proof}[Proof of Proposition \ref{prop:klconv}]
In order to clarify the boundary assumption (A4), we need to defined a GP conditioned on an infinite number of observations (e.g., all variables outside of $\inputdomain$). 
Therefore, it is not enough to treat the covariance function simply as a machine that produces covariance matrices for each finite set of inputs, but instead we have to consider the full infinite-dimensional Gaussian measures induced by GPs. As detailed in \cite{Owhadi2015}, a centered Gaussian measure $P$ on a Hilbert space $H$ is characterized by a covariance operator $\mathcal{G}:H \longrightarrow H$ such that  
\begin{equation}
    \left \langle h_1, \mathcal{G} h_2 \right \rangle = \int_H  \left \langle h_1 , \bx\right \rangle \left \langle h_2, \bx\right \rangle \mathrm{d} P(\bx), \quad ^{\forall} h_1, h_2 \in H,
\end{equation}
where $\langle \cdot, \cdot \rangle$ is the inner product on $H$.
By the Riesz representation theorem, we can equivalently consider $\mathcal{G}$ as a map from the dual $H^*$ to $H$, such that 
\begin{equation}
    \left[ \phi_1, \mathcal{G} \phi_2 \right] = \int_H  \left[\phi_1 , \bx\right] \left[ \phi_2, \bx\right] \mathrm{d} P(\bx), \quad ^{\forall} \phi_1, \phi_2 \in H^*.
\end{equation}
In applications, the operator $\mathcal{G}$ is given as the convolution with a covariance function $K: \mathbb{R}^d \times \mathbb{R}^d \longrightarrow \mathbb{R}$ and the most common measurements $\phi_{1}, \phi_{2}$ are Dirac measures in points $\bx, \bx' \in \mathbb{R}^d$, which implies that
\begin{align}
    \left[ \phi_1, \mathcal{G} \phi_2 \right] 
    &= \int_{\mathbb{R}^d} \phi_{1}(\bz) \int_{\mathbb{R}^d} K(\bz, \bz') \phi_{2}(\bz') \mathrm{d} \bz' \mathrm{d} \bz \\
    &= \int_{\mathbb{R}^d} \delta_{\bx}(\bz)\int_{\mathbb{R}^d} K(\bz, \bz') \delta_{\bx'}( \bz' ) \mathrm{d} \bz' \mathrm{d} \bz = K(\bx, \bx').
\end{align}
Gaussian measures can be conditioned on closed subspaces $S \subset H^*$, resulting in a conditional covariance operator $\mathcal{G}|S$ given by the shorted operator \citep{Owhadi2015}, which is characterized as
$\left[ \phi, \mathcal{G}|S \phi \right] = \inf_{\varphi \in S} \left[\left(\phi + \varphi\right), \mathcal{G}\left(\phi + \varphi\right)\right].$
For our purpose, the Hilbert space $H$ will be the Sobolev space $H_0^s(\mathbb{R}^d)$, which is the closure of the space of smooth and compactly supported functions on $\inputdomain$, under the norm induced by the inner product
\begin{equation}
\left\langle u, v \right \rangle_{H^s\left(\mathbb{R}^d\right)} \coloneqq \sum_{|\alpha| \leq s} \int_{\mathbb{R}^{d}} \langle D^{\alpha} u(\bx), D^{\alpha} v(\bx) \rangle \mathrm{d} \bx. 
\end{equation}
The covariance operator $\mathcal{G}$ is defined in terms of the covariance function $K$ as follows 
\begin{equation}
\left(\mathcal{G} \phi\right)(\bx) = (K * \phi)(\bx) =  \left[ \phi, K(\bx - \cdot) \right], \quad ^{\forall} \bx \in \mathbb{R}^d, \ \phi \in H^{-s}(\mathbb{R}^d). 
\end{equation} 
As is customary, we denote the dual space of $H_0^s(\mathbb{R}^d)$ (with respect to the $L^2$ duality pairing) as $H^{-s}(\mathbb{R}^d)$.
The space $S \subset H^{-s}\left(\mathbb{R}\right)$ representing conditioning on the complement of $\inputdomain$ is given by the closure of the set $\left\{\delta_{\bx}: \bx \in \mathbb{R}^d \setminus \inputdomain \right\}$ in $H^{-s}\left(\mathbb{R}^d\right)$. 

We have to show that under the assumptions (A1) - (A4), we are in the setting considered by Theorem 4.1 of \cite{Schafer2017}. As the first step, we need to show that the covariance operator $\mathcal{G}$ of a GP that satisfies (A3) and (A4) has the following properties:
\begin{enumerate}
    \item[($i$)] $\mathcal{G}$ is a linear and invertible continuous map from $H^{-s}(\inputdomain)$ to $H_0^s(\inputdomain)$. 
    \item[($ii$)] The precision operator $\mathcal{L} \coloneqq \mathcal{G}^{-1}$ is symmetric and non-negative, in the sense that 
    \begin{equation}
        \int_{\inputdomain} u \mathcal{L} v \mathrm{d} \bx = \int_{\inputdomain} v \mathcal{L} u \mathrm{d} \bx \quad \text{and} \quad \int_{\inputdomain} u \mathcal{L} u \mathrm{d} \bx \geq 0, \quad ^{\forall} u,v \in H^s_{0}(\inputdomain)
    \end{equation}
    \item[($iii$)] The precision operator is local, in the sense that 
    \begin{equation}
        \int_{\inputdomain} u \mathcal{L} v \mathrm{d} \bx = 0, \ ^{\forall} u,v \in H_0^s(\inputdomain), \quad \text{such that} \quad \operatorname{supp}(u) \cap \operatorname{supp}(v) = \emptyset.
    \end{equation}
\end{enumerate}
We begin by showing that under the assumption (A3), the operator satisfies the above properties with $\inputdomain = \mathbb{R}^d$.
We then show that after conditioning on the values of the process on $\mathbb{R}^d \setminus \inputdomain$, we recover obtain the above properties. By the assumption (A3), application of $\mathcal{G}$ to a function $u$ in Schwartz space amounts to dividing its Fourier transform by the polynomial $\psi$.
By the equivalence of polynomial multiplication in Fourier space and differentiation, the precision operator $\mathcal{L} = \mathcal{G}^{-1}$ is therefore the differential operator
\begin{equation}
    \mathcal{L} = \sum_{|\alpha| \leq s} (-1)^{|\alpha|} a_{\alpha} D^{2 \alpha} u
\end{equation}
which we consider as a mapping by from $H^s_0\left(\mathbb{R}^d\right)$ to  $H^{-s}\left(\mathbb{R}^d\right)$ as 
\begin{align}
    \left[\mathcal{L}u, v\right] 
    &\coloneqq \left \langle u, v\right \rangle_{\mathcal{L}} \\ 
    &\coloneqq \int_{\mathbb{R}^d} u(\bx) \mathcal{L} v(\bx) \mathrm{d}\bx \\
    &= \sum_{|\alpha| \leq s} a_{\alpha} \int_{\mathbb{R}^d} \langle D^\alpha u(\bx), D^{\alpha} v(\bx) \rangle \mathrm{d}\bx, \quad ^{\forall} u,v \in H_0^s\left(\mathbb{R}^d\right),
\end{align}
where the equality follows from integration by parts, implying symmetry and positivity of $\mathcal{L}$. 
According to the lower bound on $\psi$ in (A3), there exists a constant $C_{\mathcal{L}}$ depending only on $c_{\psi}$, $s$, and $d$, such that 
\begin{equation}
\frac{1}{C_{\mathcal{L}}} \|u\|_{H^s\left(\mathbb{R}^d\right)} \leq \|u\|_{\mathcal{L}} \leq C_{\mathcal{L}} \|u\|_{H^s\left(\mathbb{R}^d\right)}.
\end{equation}
Thus, $\left\langle \cdot, \cdot \right \rangle_{\mathcal{L}}$ is an inner product equivalent to the $H^s$ inner product $\left \langle \cdot, \cdot \right \rangle_{H^s\left( \mathbb{R}^d \right)}$.
By the Riesz representation theorem, it also follows that $\mathcal{L}$ is surjective.
The equivalence of $\|\cdot\|_{\mathcal{L}}$ and $\|\cdot \|_{H^s}$ implies the boundedness of both $\mathcal{L}$ and $\mathcal{G}$. 
As a differential operator, $\mathcal{L}$ is furthermore local, and thus $\mathcal{L}$ and $\mathcal{G}$ satisfy properties ($i$) - ($iii$) above with $\mathbb{R}^d$. 

We now want to show that after conditioning the Gaussian measure on the complement $\mathbb{R}^d \setminus \inputdomain$, the conditional precision operator is given by the restriction of $\mathcal{L}$ to $H_0^s\left(\inputdomain\right)$.
To this end, we use duality to rewrite the definition of the shorted operator $\mathcal{G}|S$ as 
\begin{align}
        \frac{\left[ \phi, \mathcal{G}|S \phi \right]}{2} &= \inf_{\varphi \in S}  \frac{\left[\left(\phi + \varphi\right), \mathcal{G}\left(\phi + \varphi\right)\right]}{2} \\
        &= \inf_{\varphi \in S} \sup_{u \in H} \left[\left(\phi + \varphi\right), u\right] - \frac{\left[\mathcal{L}u, u\right]}{2} \\
        &\geq \sup_{u \in H} \inf_{\varphi \in S}  \left[\left(\phi + \varphi\right), u\right] - \frac{\left[\mathcal{L}u, u\right]}{2}.
\end{align}
At the same time,  we also have 
\begin{align}
        \frac{\left[ \phi, \mathcal{G}|S \phi \right]}{2}  
        &= \inf_{\varphi \in S} \left[\left(\phi + \varphi\right), \mathcal{G}\left(\phi + \varphi\right)\right] - \frac{\left[\mathcal{L}\mathcal{G}\left(\phi + \varphi\right), \mathcal{G}\left(\phi + \varphi\right)\right]}{2} \\
        &\leq 
        \sup_{u \in H} \inf_{\varphi \in S} \left[\left(\phi + \varphi\right), u\right] - \frac{\left[\mathcal{L}u, u\right]}{2}.
\end{align}
Together, we have
\begin{equation}
        \frac{\left[ \phi, \mathcal{G}|S \phi \right]}{2} =  
        \sup_{u \in H} \inf_{\varphi \in S}  \left[\left(\phi + \varphi\right), u\right] - \left[\mathcal{L}u, u\right] 
        = \sup_{u \in S^{\perp}} \left[\phi, u\right] - \frac{\left[\mathcal{L}u, u\right]}{2}, 
\end{equation}
where $S^{\perp}$ denotes the orthogonal complement of $S$.
Since $s > d$, $H^s\left(\mathbb{R}^d\right)$ embeds into the continuous functions and therefore $S^{\perp}$ is given by exactly those function in $H^{s}\left(\mathbb{R}^d\right)$ that are equal to zero, outside of $\inputdomain$. By Theorem 5.11 of \cite{Brewster2014}, these are exactly the extensions with zero of functions in $H^s_0\left(\inputdomain\right)$. 
For $\phi \in H^{-s}\left(\inputdomain\right) = \left(H_{0}\left(\inputdomain\right)\right)^*$, this results in the expression 
\begin{equation}
        \frac{\left[ \phi, \mathcal{G}|S \phi \right]}{2} = \sup_{u \in H_0\left(\inputdomain\right)} \left[\phi, u\right] - \left[\mathcal{L}u, u\right] = \left(\mathcal{L}\big|_{H_0^s\left(\inputdomain\right)}\right)^{-1} \phi,
\end{equation}
where $\mathcal{L}\big|_{H_0\left(\inputdomain\right)}: H_0^s\left(\inputdomain\right) \longrightarrow H^{-s}\left(\inputdomain\right)$ is the restriction of $\mathcal{L}$ to $H_{0}\left(\inputdomain\right)$.
Since $\mathcal{L}\big|_{H_0\left(\inputdomain\right)}$ inherits locality, symmetry, and positivity from $\mathcal{L}$, we have thus shown that the conditional covariance operator $\mathcal{G}|S$ satisfies the assumptions of Theorem 4.1 in \cite{Schafer2017} and therefore also of Theorem 3.4 in \cite{Schafer2020}, which is based on the former.
According to Theorem 3.4 of \cite{Schafer2017}, there exists a constant $c > 0$ depending on $d$, $\inputdomain$, $s$, $\|\mathcal{L}\|$ and $\|\mathcal{L}^{-1}\|$, such that for $\rho \geq c \log(n / \epsilon)$, 
\begin{equation}
   \KL(P_n\|\hat{P}_{n,\rho}) \leq \epsilon,
\end{equation}
where $P_n$ is the Gaussian measure with the multivariate normal density $\dens_{n} (\cdot , \bftheta)$ under (A4) and $\hat{P}_{n, \rho}$ is the nearest-neighbor-based Vecchia approximation of $P_n$ computed by using the radius-based conditioning sets $s_i = \{j \leq i: \| \bx_i - \bx_j \| \leq \rho \ell_i\}$, for $i = 1, \ldots , n$ \citep{Schafer2020}.
Here, the length scale $\ell_i$ is defined as $\ell_i = \min_{j < i} \| \bx_i - \bx_j \|$. 
That is, the length scale of $\bx_i$ is given by the distance to the closest point $\bx_j$ appearing before $\bx_i$ in the maximin ordering. 
All points appearing before $\bx_i$ in the maximin ordering have at least distance $\ell_i$ from one another. Furthermore, they have to cover $\inputdomain$ up to balls of size $(1 + \delta_n) \ell_i$, since otherwise the maximin ordering would have chosen another point instead of $\bx_i$. 
Thus, there exists a constant $c_{\delta_n} > 0$ depending only on $\delta_n$, such that for $m > c_{\delta_n} \rho^d$, the (size-based) conditioning set is a superset of the radius-based conditioning set $s_i$. 
Since the KL divergence only decreases upon increasing the conditioning set,  

$$
^{\exists}\, c>0 \text{ such that } \KL(P_n\|\hat{P}_{n,m}) \leq \epsilon \text{ for } m \geq c \log^d(n/\epsilon),
$$
where $\hat{P}_{n,m}$ is the Gaussian measure with the Vecchia likelihood $\adens_{n,m} (\cdot , \bftheta)$ under (A4). Thus, for $\epsilon = n^{-k}$ with any fixed $k \in \mathbb{Z}^+$, we have
$$
^{\exists}\, c>0 \text{ such that } \KL(P_n\|\hat{P}_{n,m}) \leq 1/n^k \text{ for } m \geq 
(k+1)^d c \log^d(n),
$$
where $(k+1)^d$ can be absorbed into the constant. Hence, we have 
$$
m = \knuth(\log^d n) \; \Rightarrow \; \KL(P_n\|\hat{P}_{n,m}) \rightarrow 0 \; \text{ as } \; n \rightarrow \infty,
$$
where convergence is faster than a polynomial of any fixed order.
\end{proof}

\begin{lemma}\label{lemm:regbound}
Assume (A1) - (A4). For any $\bftheta \in \bfTheta$,
\begin{equation}
    m = \knuth(\log^d n) \; \Rightarrow \; \KL \Big( \dens_{\bftheta} (y_i | \by_{h(i)}) \,\Big\|\, \dens_{\bftheta} (y_i | \by_{g(i)}) \Big) \rightarrow 0 \; \text{ as } \; n \rightarrow \infty,
\end{equation}
for all $i = 1, \ldots, n$, where the KL divergence is defined between Gaussian measures with densities $\dens_{\bftheta} (y_i | \by_{h(i)})$ and $\dens_{\bftheta} (y_i | \by_{g(i)})$ under (A4).
\end{lemma}

\begin{proof}[Proof of Lemma \ref{lemm:regbound}]
The exact and Vecchia log-likelihoods can be written as
$$
\log \dens_{n} (\by ; \bftheta) 
= \sum_{i=1}^n \log \dens_{\bftheta} (y_i|\by_{h(i)}) \quad\text{and}\quad \log \adens_{n,m} (\by ; \bftheta) 
= \sum_{i=1}^n \log \adens_{\bftheta} (y_i|\by_{h(i)}),$$
where $\adens_{\bftheta} (y_i|\by_{h(i)}) = \dens_{\bftheta} (y_i|\by_{g(i)})$ for $i = 1, \ldots , n$. The KL divergence between Gaussian measures according to the densities $\dens_{n} (\by ; \bftheta)$ and $\adens_{n,m} (\by ; \bftheta)$ can be expressed as
\begin{align*}
    \KL \Big( \dens_{n} (\by ; \bftheta) \,\Big\|\, \adens_{n,m} (\by ; \bftheta) \Big) 
    &= \int \dens_{n} (\by ; \bftheta) \log\left[\frac{\dens_{n} (\by ; \bftheta)}{\adens_{n,m} (\by ; \bftheta)}\right] \text{d}\by\\
    &= \sum_{i=1}^{n}\int \dens_{n} (\by ; \bftheta) \log\left[\frac{\dens_{\bftheta}(y_i | \by_{h(i)})}{\adens_{\bftheta} (y_i | \by_{h(i)})}\right] \text{d}\by \\
    &= \sum_{i=1}^{n} \int \KL \Big( \dens_{\bftheta} (y_i | \by_{h(i)}) \,\Big\|\, \adens_{\bftheta} (y_i | \by_{h(i)}) \Big) \text{d} \dens_{\bftheta} (\by_{h(i)}) \\
    &= \sum_{i=1}^{n} \mathbb{E}_{\bftheta} \left[ \KL \Big( \dens_{\bftheta} (y_i | \by_{h(i)}) \,\Big\|\, \adens_{\bftheta} (y_i | \by_{h(i)}) \Big) \right]
\end{align*}
Since $\KL \Big( \dens_{\bftheta} (y_i | \by_{h(i)}) \,\Big\|\, \adens_{\bftheta} (y_i | \by_{h(i)}) \Big)$ is always non-negative, 
$$
\KL \Big( \dens_{n} (\by) \,\Big\|\, \adens_{n,m} (\by) ) \Big) \longrightarrow 0 \; \text{ as } \; n \longrightarrow \infty
$$
implies that
$$
\KL \Big( \dens_{\bftheta} (y_i | \by_{h(i)}) \,\Big\|\, \adens_{\bftheta} (y_i | \by_{h(i)}) \Big) \longrightarrow 0 \; \text{ as } \; n \longrightarrow \infty
$$
for every $i \in \mathbb{Z}^+$. Note that the conditional KL divergence equals to zero when both the conditional densities are exactly equal almost everywhere on the support of the first density. Since all the densities we consider are Gaussian densities, their supports are always the same.
\end{proof}

\begin{lemma}\label{lemm:unifconv}
Assuming (A1) - (A8), we have: 
\begin{equation}
    m = \knuth(\log^d n) \quad \Rightarrow \quad \sup_{\bftheta \in \bfTheta} | \log \dens_{n} (\cdot ; \bftheta) - \log \hat{\dens}_{n,m} (\cdot ; \bftheta) | \overset{p}{\to} 0 \, \text{ under } P_{\bftheta_0}
\end{equation}
where $P_{\bftheta_0}$ denote the Gaussian measure corresponding to the GP with the data-generating parameter $\bftheta_0 \in \bfTheta$.
\end{lemma}

\begin{proof}[Proof of Lemma~\ref{lemm:unifconv}] 
Let $l(\bftheta) = \log \dens_{n} (\cdot ; \bftheta) - \log \hat{\dens}_{n,m} (\cdot ; \bftheta)$. If suffices to show that $l(\bftheta)$ converges to 0 in probability for every $\bftheta \in \bfTheta$ and is equicontinuous in probability. First, from (A1) - (A4), we know that the KL divergence between $\dens_{n} (\by ; \bftheta)$ and $\hat{\dens}_{n,m} (\by ; \bftheta)$ converges to zero as $m = \knuth(\log^d n)$ and $n \rightarrow \infty$. Then,
\begin{equation}
    \mathbb{E}_{\bftheta} \Big[ \dens_{n} (\by ; \bftheta) - \hat{\dens}_{n,m} (\by ; \bftheta) \Big] = \KL \Big( \dens_{n} (\cdot ; \bftheta) \| \hat{\dens}_{n,m} (\cdot ; \bftheta) \Big) \rightarrow 0
\end{equation}
as $m = \knuth(\log^d n)$ and $n \rightarrow \infty$. Also,
\begin{align}
    \text{var}_{\bftheta} \Big[ \dens_{n} (\by ; \bftheta) - \hat{\dens}_{n,m} (\by ; \bftheta) \Big] 
    &= \frac{1}{4} \text{var}_{\bftheta} \Big[ \by^{\top} (\bK_n^{-1} - \hat{\bK}_{n,m}^{-1}) \by \Big]\\
    &= \frac{1}{2} \trace \Big[ (\bI_n - \hat{\bK}_{n,m}^{-1} \bK_n)^2 \Big]
    = \frac{1}{2} \sum_{i=1}^{n} (1 - \lambda_{n,i})^2.
\end{align}
where $\bI_n$ is the $n$ by $n$ identity matrix and $\lambda_{n,1} , \ldots , \lambda_{n,n}$ are eigenvalues of $\hat{\bK}_{n,m}^{-1} \bK_n$. Applying Pinsker's inequality, one can see that total variation between Gaussian measures induced by $\dens_{n} (\by ; \bftheta)$ and $\hat{\dens}_{n,m} (\by ; \bftheta)$ also tends to zero. Also, from Theorem 1.1 of \cite{Devroye2018}, $\sum_{i=1}^{n} (1-\lambda_{n,i})^2 \rightarrow 0$ as $m = \knuth(\log^d n)$ and $n \rightarrow \infty$. By applying Markov's inequality, the pointwise convergence can be achieved.

To show that $l(\bftheta)$ is equicontinuous in probability, we use the mean value theorem as follows: Since
\begin{align}
    | l(\bftheta_1) - l(\bftheta_2)| \le \sup_{\bftheta \in \bfTheta} \| \nabla l(\bftheta) \| \| \bftheta_1 - \bftheta_2 \|
\end{align}
for any $\bftheta_1 , \bftheta_2 \in \bfTheta$, it suffices to show that $\sup_{\bftheta \in \bfTheta} \| \nabla l(\bftheta) \|$ is bounded in probability. From (A8), one can see that
\begin{align}
    \mathbb{E}_{\bftheta} \Big[ \sup_{\bftheta \in \bfTheta} \| \nabla l(\bftheta) \| \Big] 
    &\le \mathbb{E}_{\bftheta_0} \Big[ \sup_{\bftheta \in \bfTheta} \| \frac{\partial}{\partial \bftheta} \log \dens_n(\by ; \bftheta) \| + \sup_{\bftheta \in \bfTheta} \| \frac{\partial}{\partial \bftheta} \hat{\dens}_{n,m}(\by ; \bftheta) \| \Big] \\
    &< c_\nabla < \infty.
\end{align}
Then, $\sup_{\bftheta \in \bfTheta} \| \nabla l(\bftheta) \|$ is bounded in probability, which implies that $l(\bftheta)$ is equicontinuous in probability.
\end{proof}

\begin{proof}[Proof of Proposition \ref{prop:mvlebound}]
    From Lemma \ref{lemm:unifconv}, there exist ($P_{\bftheta_0}$-almost surely) a sequence $m = \knuth(\log^d n)$ such that $$ \hat{\bftheta}_{n,m} - \hat{\bftheta}_n \rightarrow 0\, \text{ in probability under } P_{\bftheta_0}$$ as $n \rightarrow \infty$. Then, from the consistency of $\hat{\bftheta}_n$, one can easily see that $\hat{\bftheta}_{n,m} \rightarrow \bftheta_0 \, \text{ in probability under } P_{\bftheta_0}.$ Also, by applying Slutsky's theorem, it can be shown that the asymptotic normality of $\hat{\bftheta}_n$ implies that of $\hat{\bftheta}_{n,m}$. Additionally, if a sequence $m'$ increases faster than $m$, the same results hold. 
\end{proof}

\begin{proof}[Proof of Proposition \ref{prop:unbias}]
The Vecchia log-likelihood can be written as
\begin{align}
    \log \adens_{n,m} (\by , \bftheta) 
    &= \sum_{i=1}^{n} \log p_{\bftheta} (y_i | \by_{g(i)}) \\
    &= -\frac{n}{2} \log(2\pi) - \frac{1}{2} \sum_{i=1}^{n} \Big[ \log \sigma^2_{i | g(i)} + {u^2_{i|g(i)}}/{\sigma^2_{i | g(i)}} \Big]
\end{align}
where $\sigma^2_{i | g(i)} = \bK_{i,i} - \bK_{i,g(i)} \bK_{g(i),g(i)}^{-1} \bK_{g(i),i}$ and $u_{i|g(i)} = y_i - \bK_{i,g(i)} \bK_{g(i),g(i)}^{-1} \by_{g(i)}$. Note that $u_{1|g(1)} = y_1$. Also $u_{i|g(i)} = y_i - \mathbb{E}_{\theta} [y_i | \by_{g(i)}]$ for $i = 2, \ldots , n$. Then, it can be easily shown that, for any $i = 1, \ldots , n$, $\mathbb{E}_{\bftheta} [u_{i|g(i)}] = 0$, $\mathbb{E}_{\bftheta} [u^2_{i|g(i)}] = \sigma^2_{i | g(i)}$, and
\begin{align}
    \text{cov}_{\bftheta} \Big[ u_{i|g(i)} , \ba^{\top} \by_{g(i)} \Big] 
    &= \mathbb{E}_{\bftheta} \Big[ u_{i|g(i)} \by_{g(i)}^{\top}\ba \Big] \\
    &= \mathbb{E}_{\bftheta} \Big[ y_i \by_{g(i)}^{\top} \Big] \ba - \bK_{i,g(i)} \bK_{g(i),g(i)}^{-1} \mathbb{E}_{\bftheta} \Big[ \by_{g(i)} \by_{g(i)}^{\top} \Big] \ba \\
    &= \bK_{i,g(i)} \ba - \bK_{i,g(i)} \bK_{g(i),g(i)}^{-1} \bK_{g(i),g(i)} \ba \\
    &= \bK_{i,g(i)} \ba - \bK_{i,g(i)} \ba = 0
\end{align}
for any column vector $\ba$ which has compatible size. This implies that $u_{i|g(i)}$ is independent of any contrast of elements of $\by_{g(i)}$.

For any $k = 1, \ldots , q$, the partial derivative of $\log \adens_{n,m} (\by , \bftheta)$ with respect to the $k$th entry of $\bftheta$ can be expressed as 
\begin{equation}
    \frac{\partial \log \adens_{n,m} (\by , \bftheta) }{\partial \theta_k} = -\frac{1}{2} \sum_{i=1}^{n} \Big[ \frac{1}{\sigma^2_{i|g(i)}} \frac{\partial \sigma^2_{i|g(i)}}{\partial \theta_k} + 2 \frac{u_{i|g(i)}}{\sigma^2_{i|g(i)}} \frac{\partial u_{i|g(i)}}{\partial \theta_k} - \frac{u^2_{i|g(i)}}{\sigma^4_{i|g(i)}} \frac{\partial \sigma^2_{i|g(i)}}{\partial \theta_k} \Big].
\end{equation}
Since $\frac{\partial u_{i|g(i)}}{\partial \theta_k}$ is a contrast of elements of $\by_{g(i)}$, it is independent of $u_{i|g(i)}$. Then,
\begin{align}
    \mathbb{E}_{\theta} \Big[ \frac{\partial \log \adens_{n,m} (\by , \bftheta) }{\partial \theta_k} \Big]
    &= -\frac{1}{2} \sum_{i=1}^{n} \Big\{ \frac{1}{\sigma^2_{i|g(i)}} \frac{\partial \sigma^2_{i|g(i)}}{\partial \theta_k} - \frac{1}{\sigma^4_{i|g(i)}} \frac{\partial \sigma^2_{i|g(i)}}{\partial \theta_k} \Big\}\\
    &= 0
\end{align}
Therefore, the Vecchia-score function is unbiased for any ordering and conditioning.
\end{proof}


\section{Additional numerical results \label{app:numeric}}

This section contains additional numerical results which are not included in the main manuscript. Complementing Figure \ref{fig:comparison_mle}, Figure \ref{fig:comparison_mle_boxplot} shows box plots of estimates and $90\%$ intervals of maximum approximate-likelihood estimators based on taper, reduced rank, and Vecchia approximations. Also, as a baseline, $90\%$ intervals of maximum likelihood estimators based on exact GP inference are plotted.

\begin{figure}
    \centering
    \includegraphics[width = 0.33\linewidth]{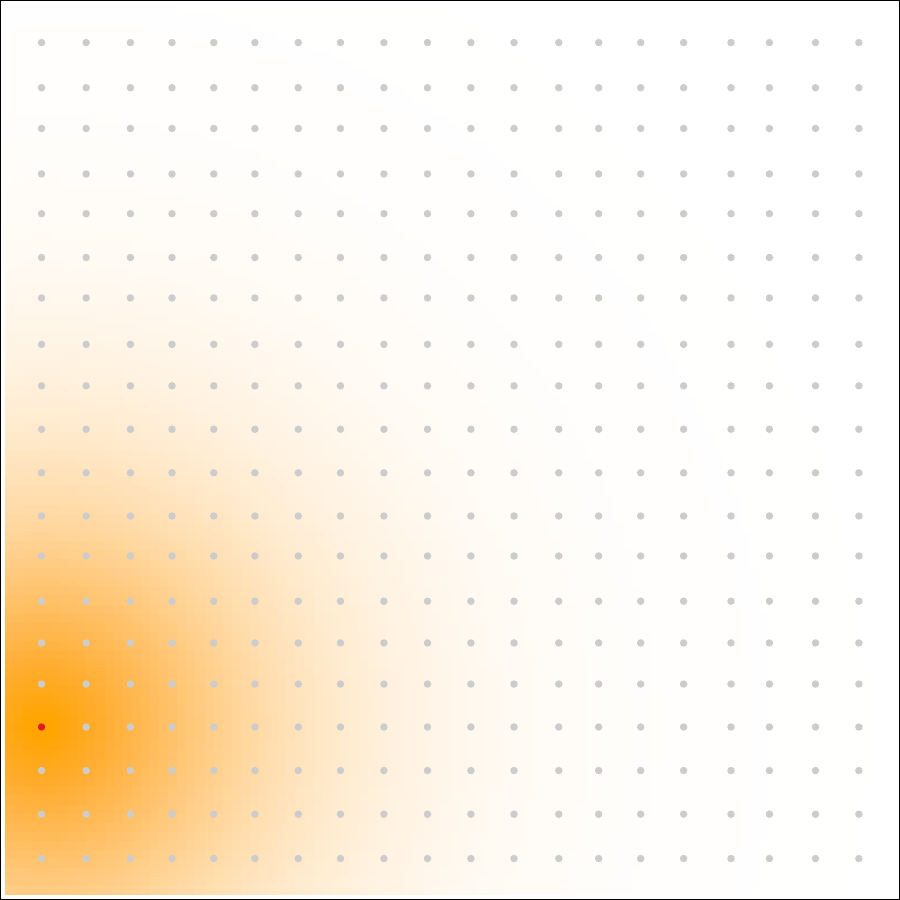}%
    \includegraphics[width = 0.33\linewidth]{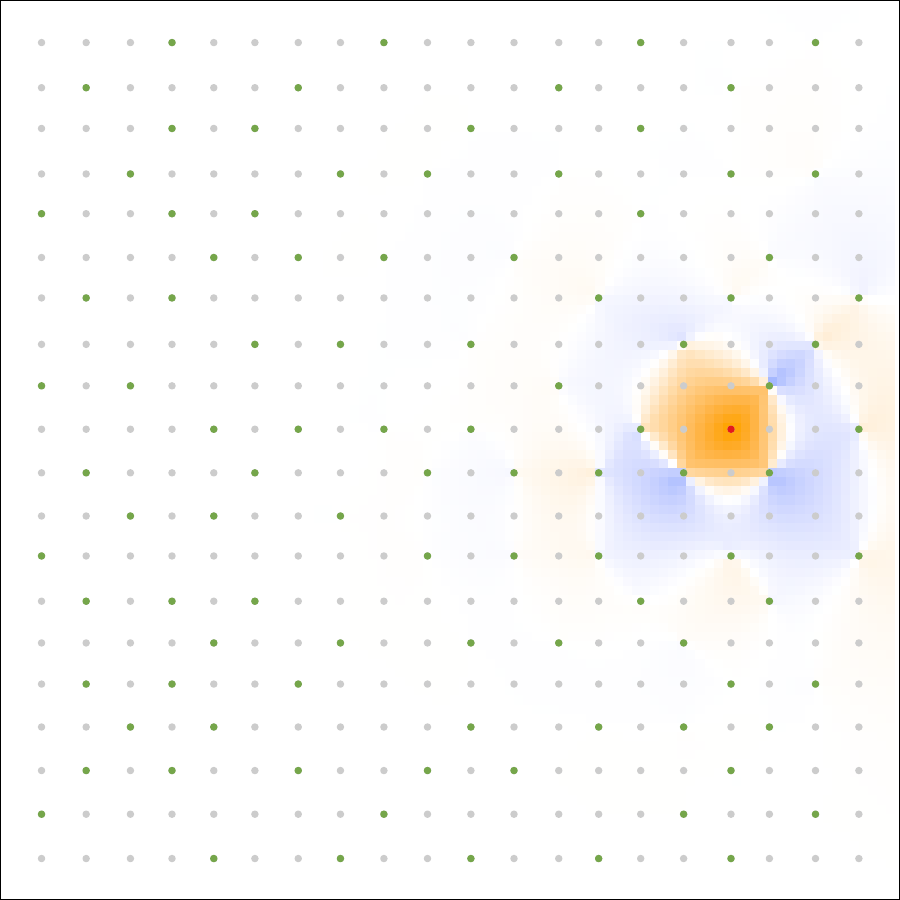}%
    \includegraphics[width = 0.33\linewidth]{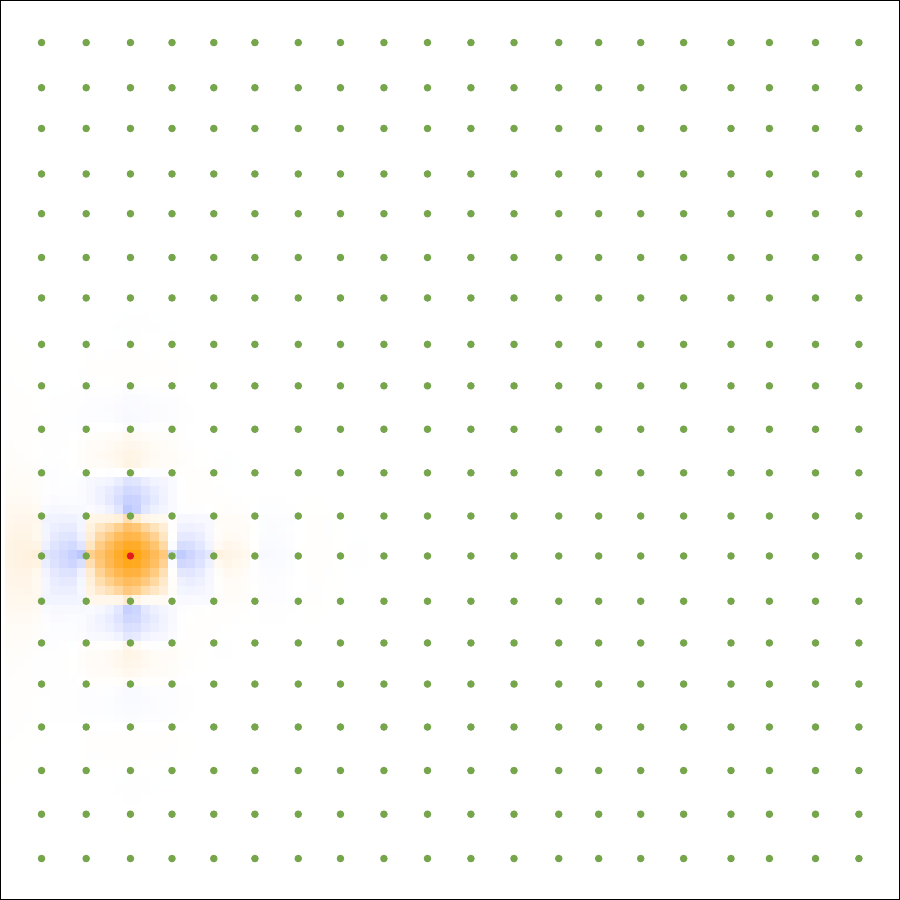}%
    \vspace{0.5em}
    \includegraphics[width = 0.33\linewidth]{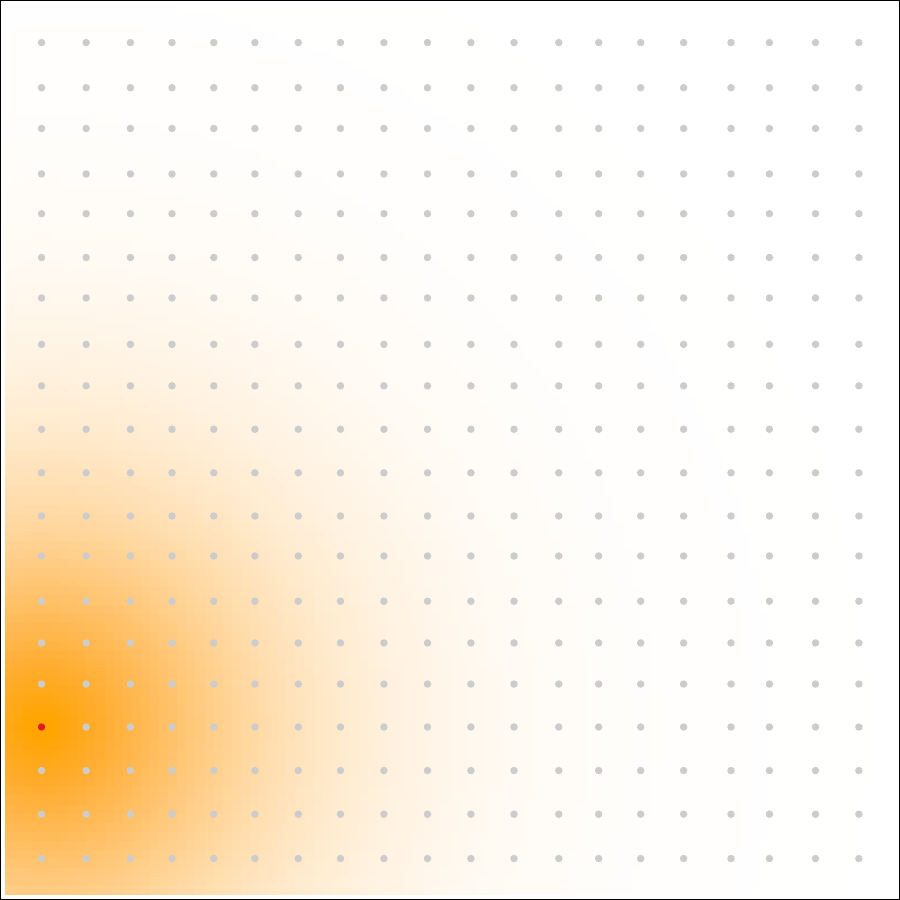}%
    \includegraphics[width = 0.33\linewidth]{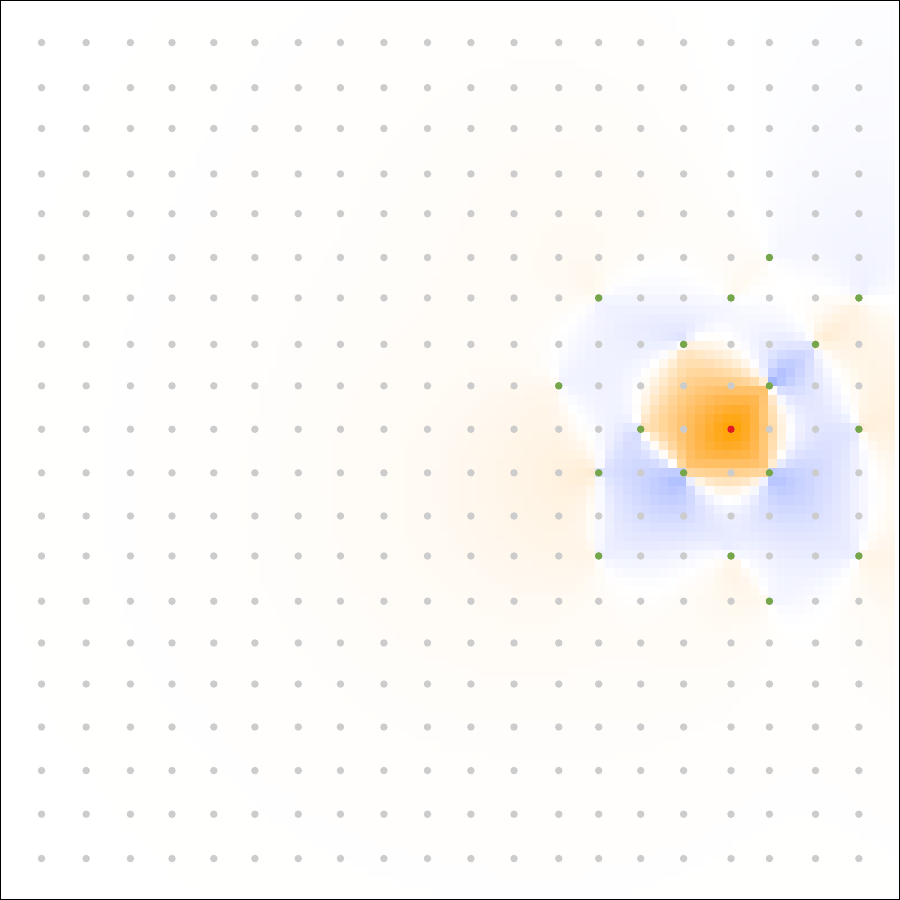}%
    \includegraphics[width = 0.33\linewidth]{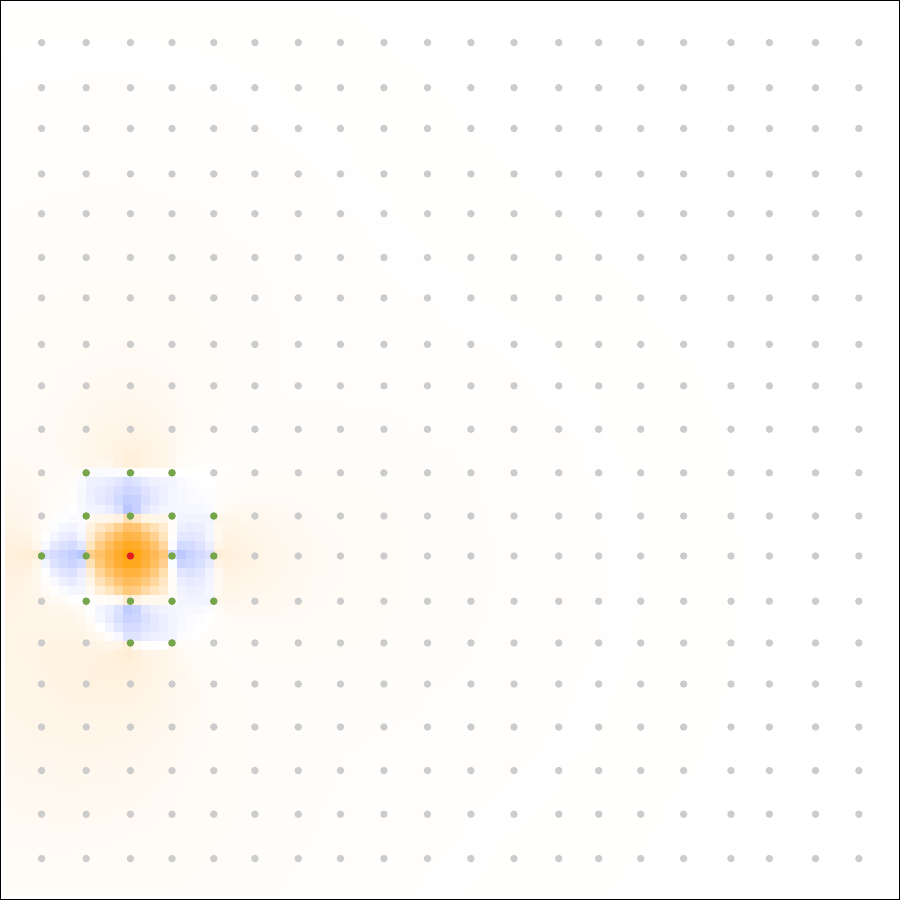}%
    \caption{For i = 1, 100, 400 (from left to right), positive (orange) and negative (blue) conditional correlations with the $i$th input $\bx_i$ in the maximin ordering (red points) conditional on all (top) and $m=17$ nearest (bottom) previously ordered inputs (green points), for a GP with Mat\'ern covariance (range $r = 0.1$ and smoothness $\nu = 2$) on a grid of size $n = 20 \times 20 = 400$. This figure is inspired by Figure 5 in \citet{Schafer2020}.}
    \label{fig:corcond_n400}
\end{figure}

\begin{figure}
    \centering
    \includegraphics[width = 0.14\linewidth]{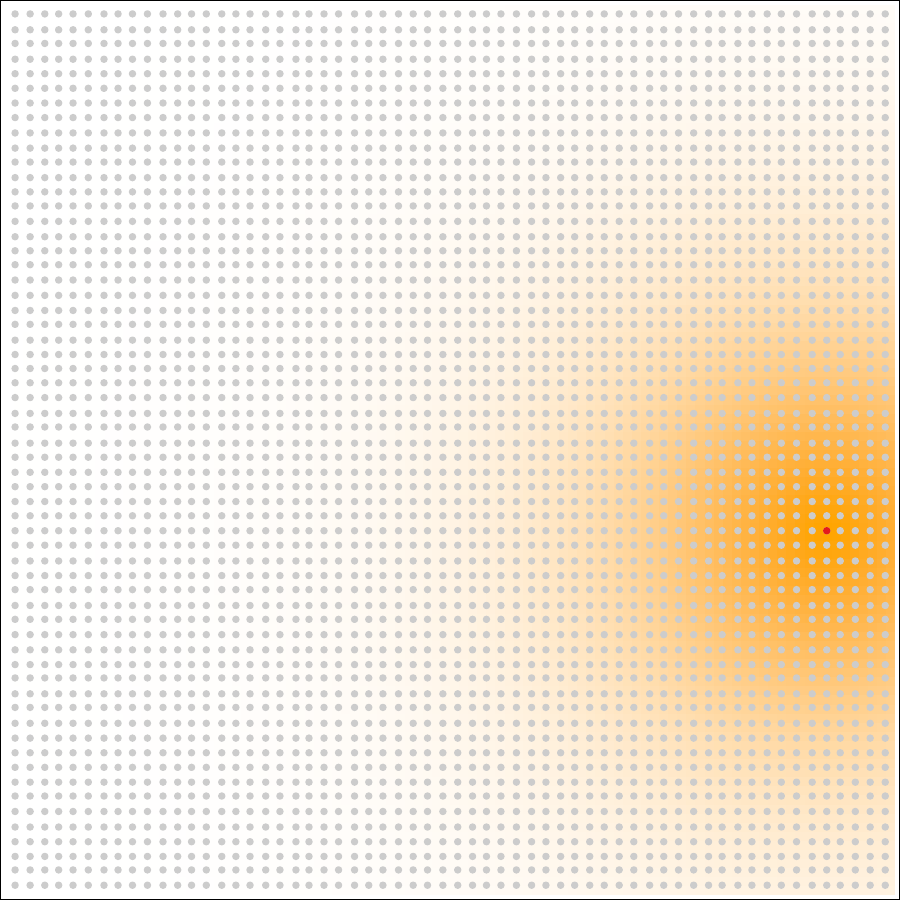}%
    \includegraphics[width = 0.14\linewidth]{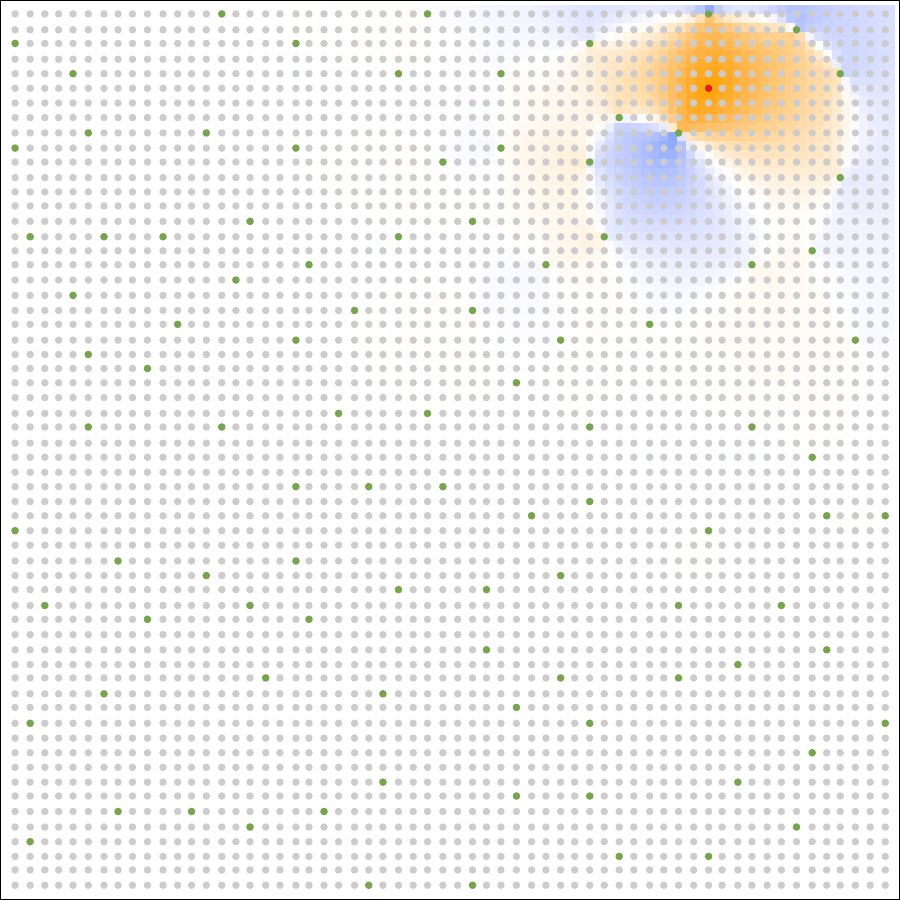}%
    \includegraphics[width = 0.14\linewidth]{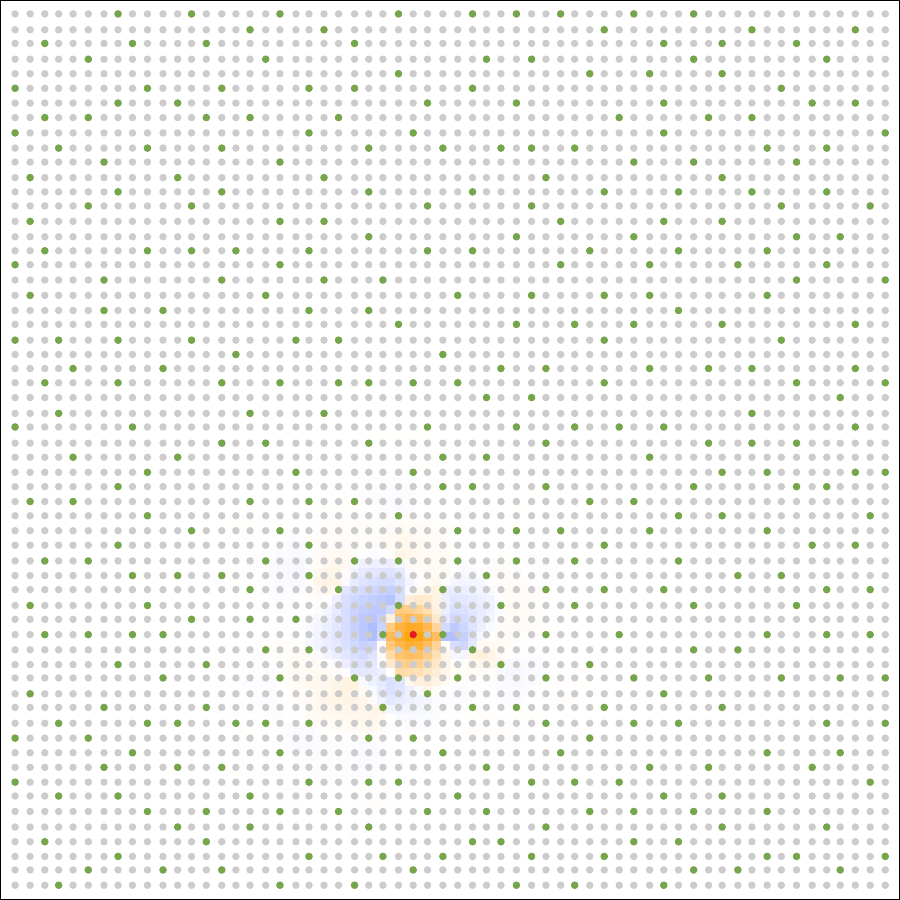}%
    \includegraphics[width = 0.14\linewidth]{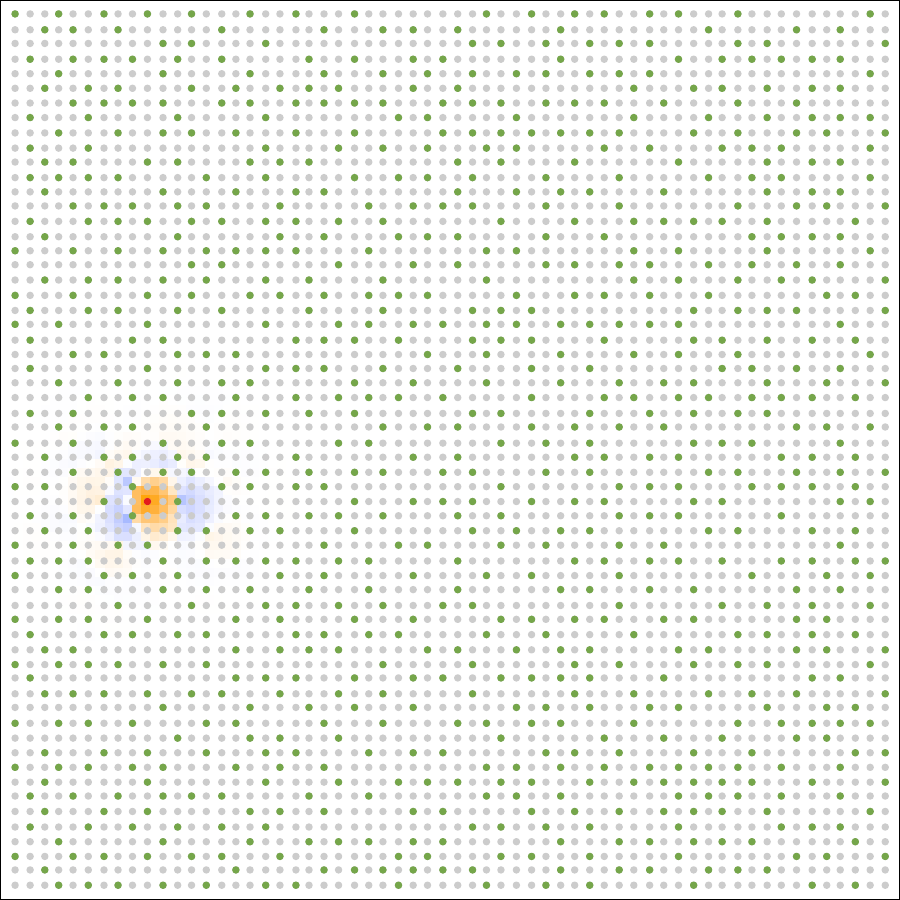}%
    \includegraphics[width = 0.14\linewidth]{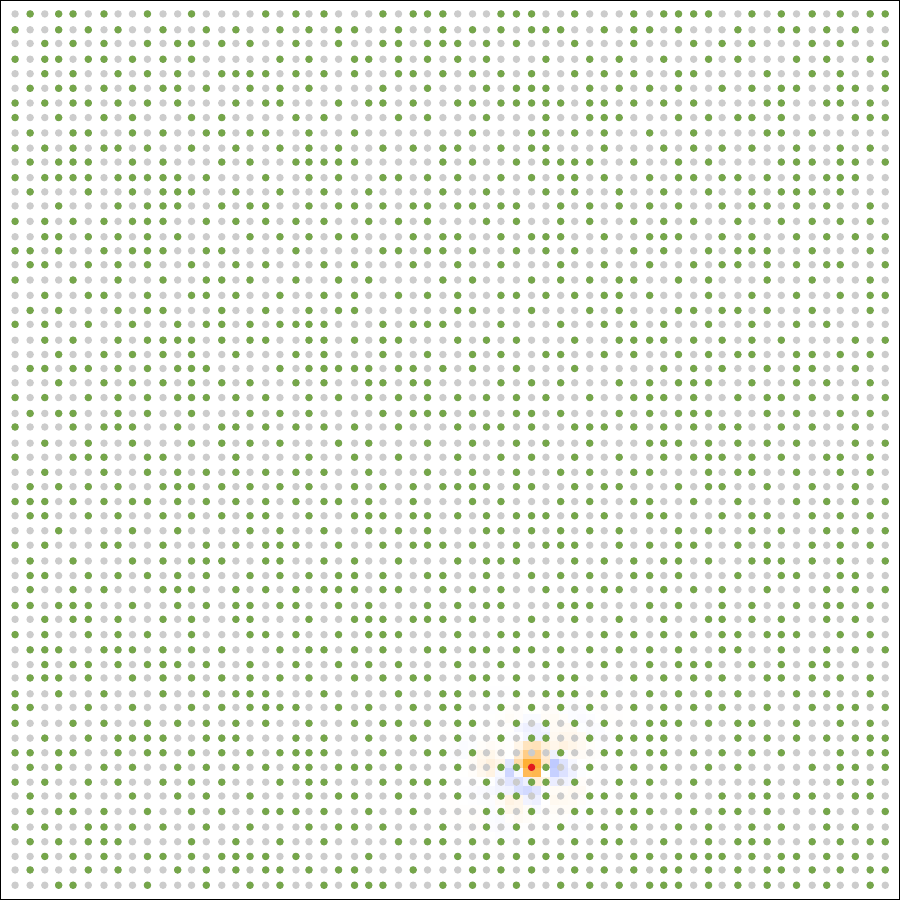}%
    \includegraphics[width = 0.14\linewidth]{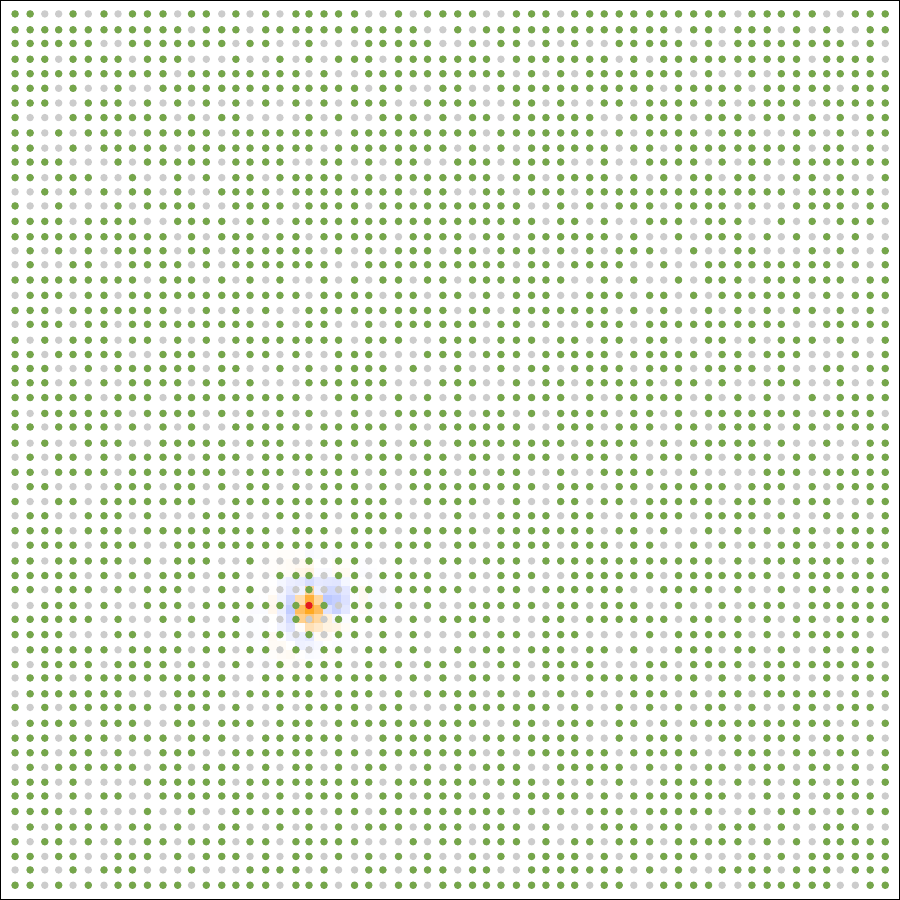}%
    \includegraphics[width = 0.14\linewidth]{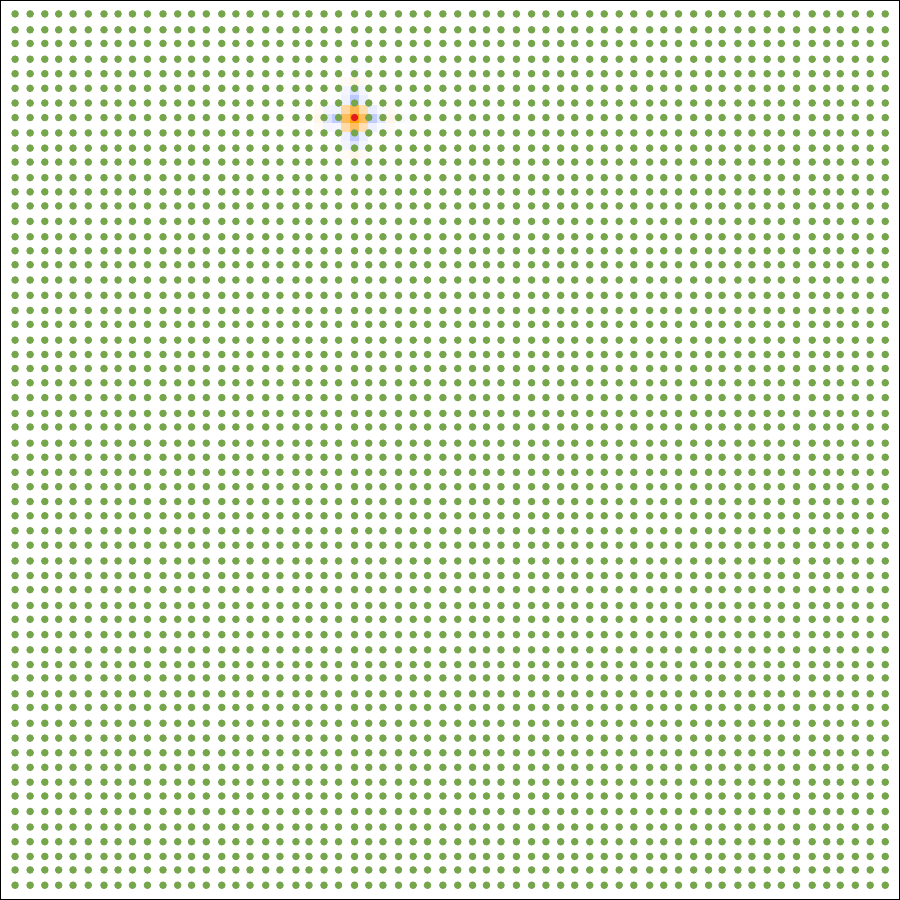}%
    \vspace{0.5em}
    \includegraphics[width = 0.14\linewidth]{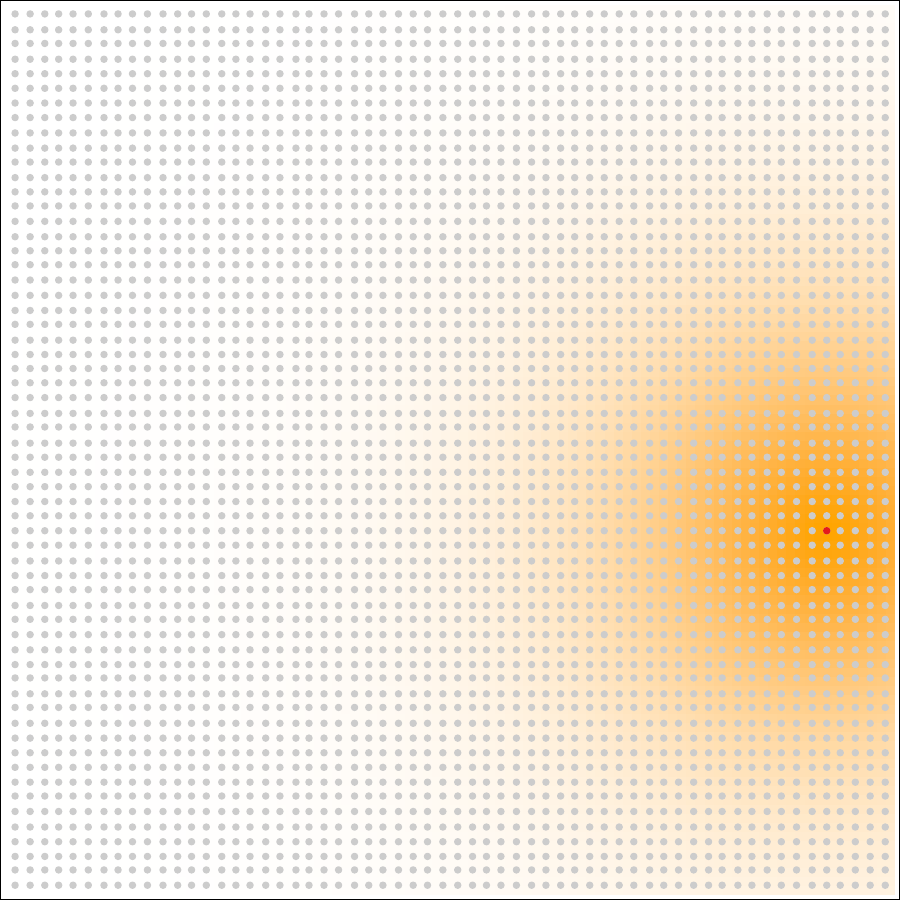}%
    \includegraphics[width = 0.14\linewidth]{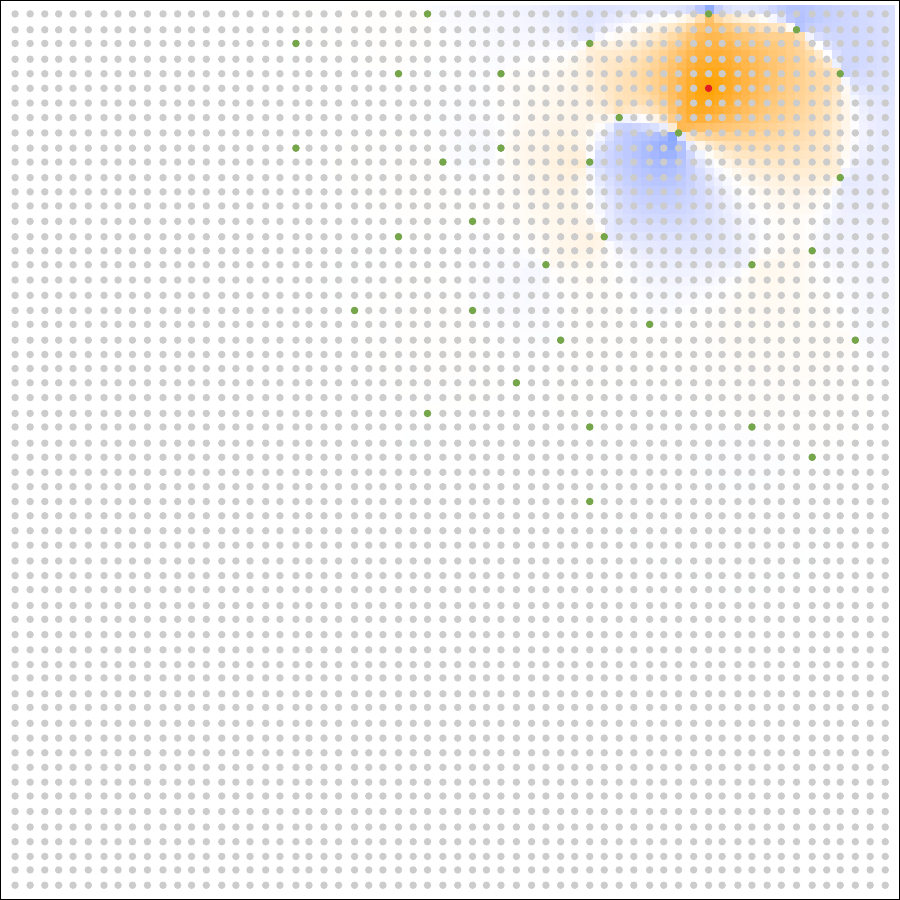}%
    \includegraphics[width = 0.14\linewidth]{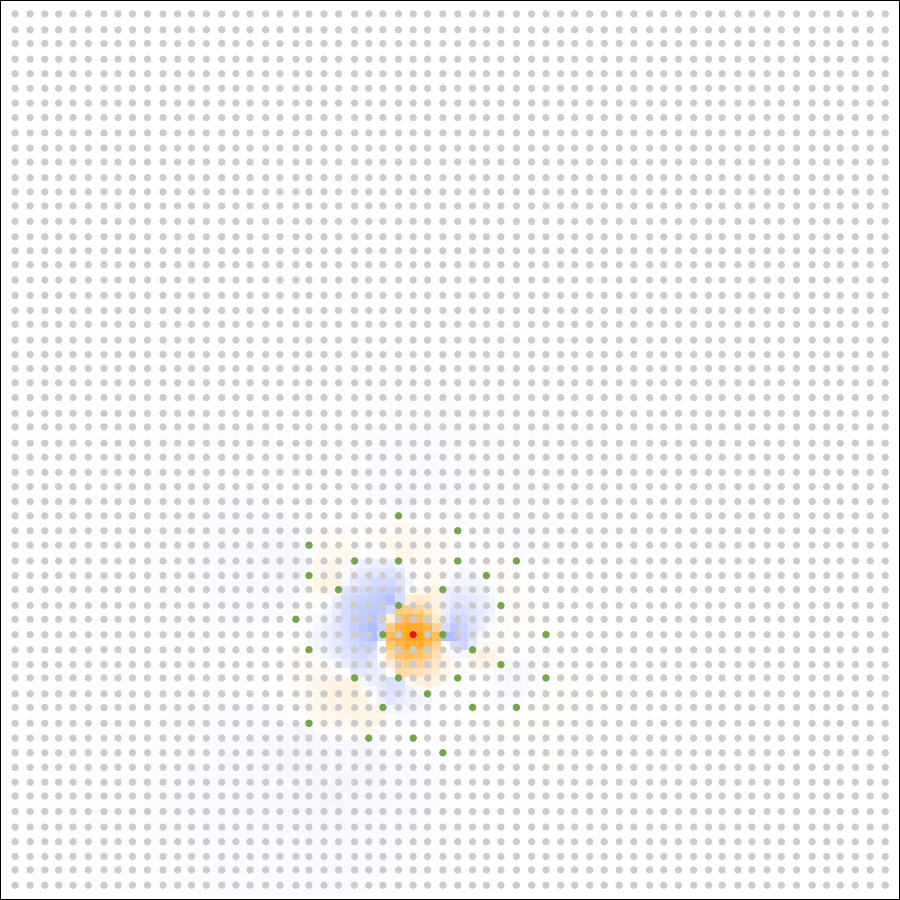}%
    \includegraphics[width = 0.14\linewidth]{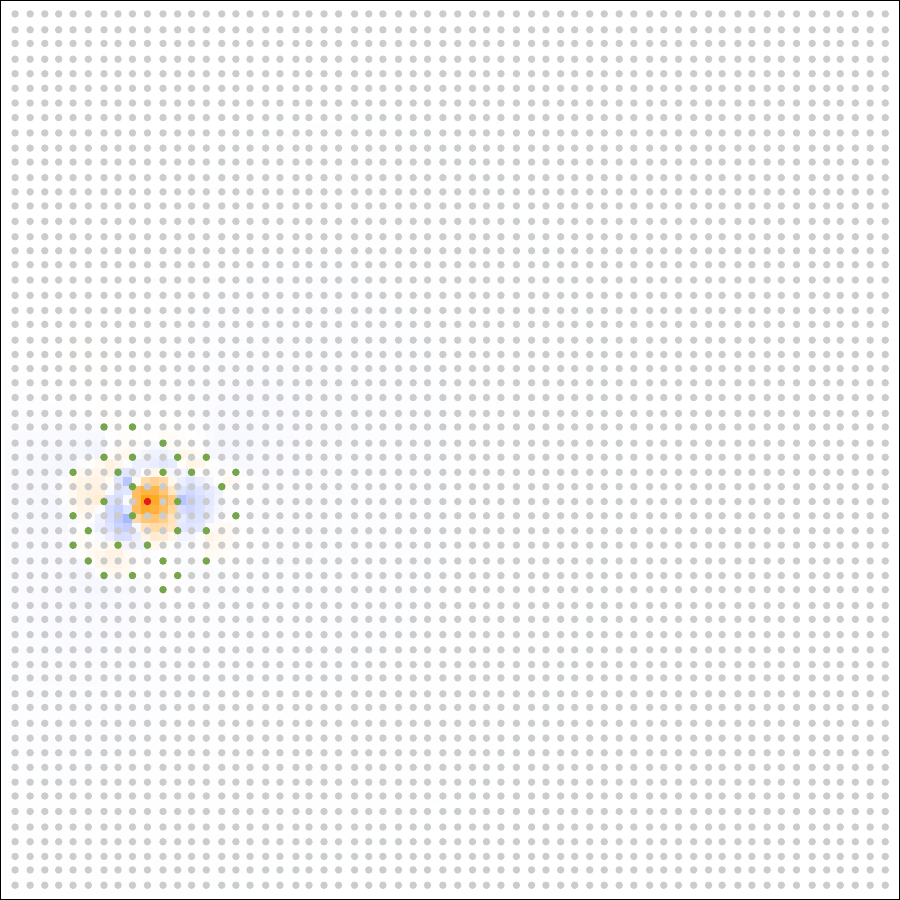}%
    \includegraphics[width = 0.14\linewidth]{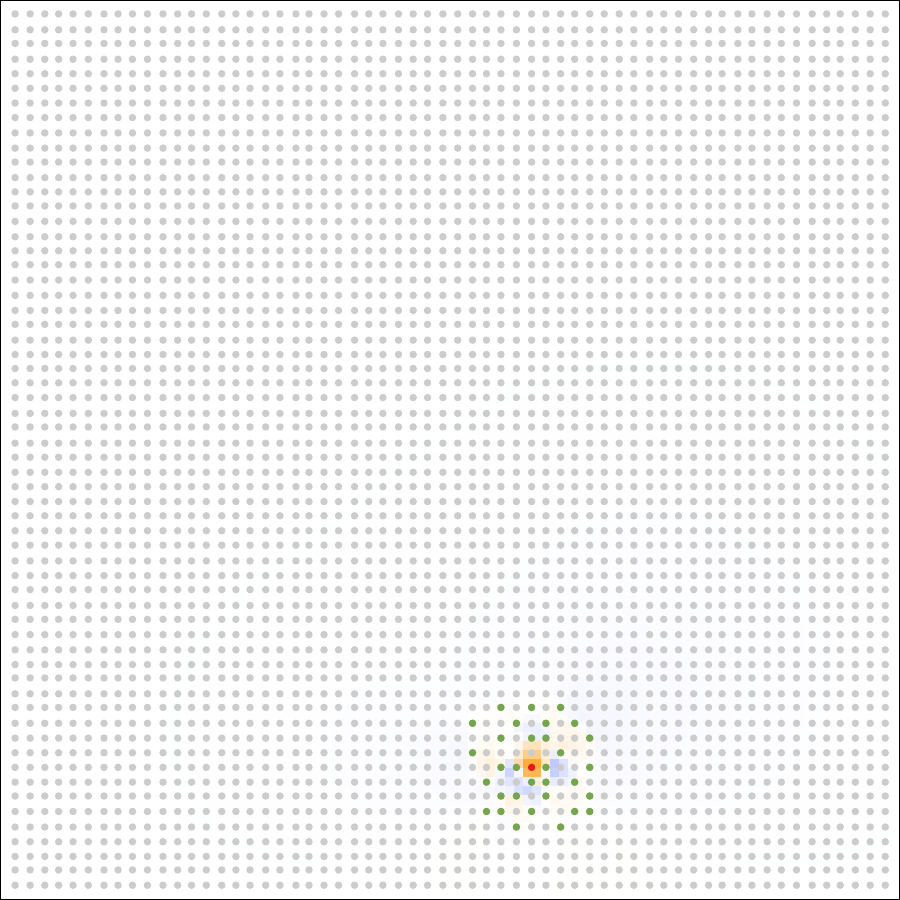}%
    \includegraphics[width = 0.14\linewidth]{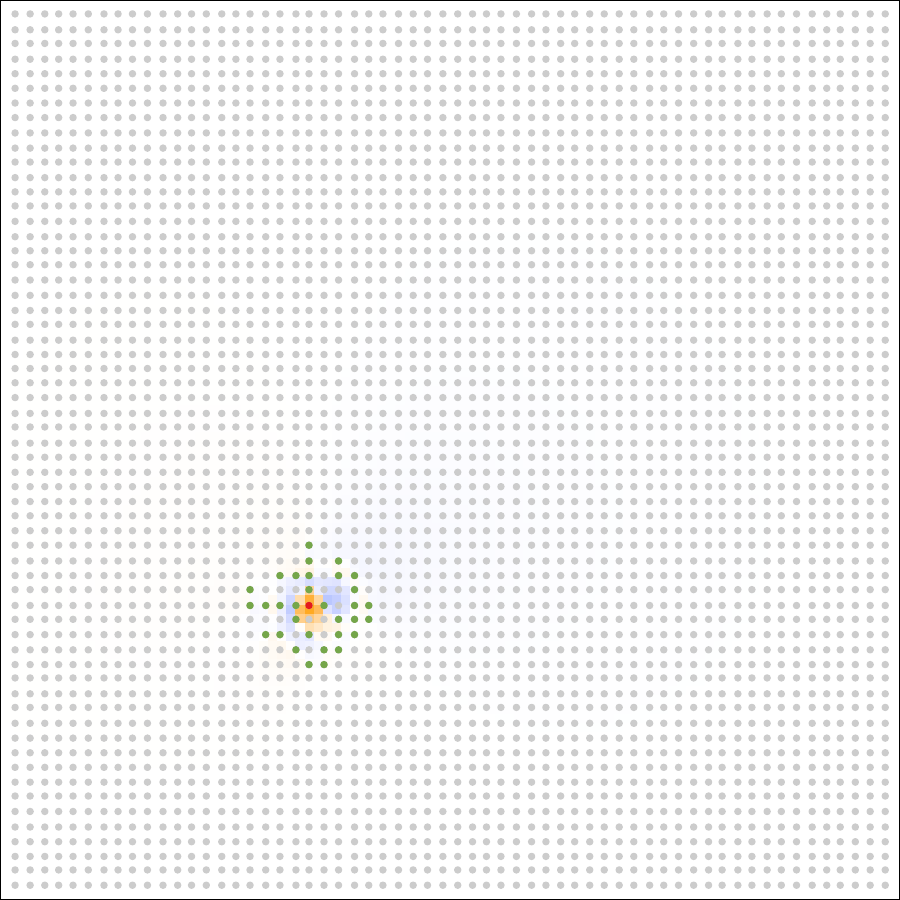}%
    \includegraphics[width = 0.14\linewidth]{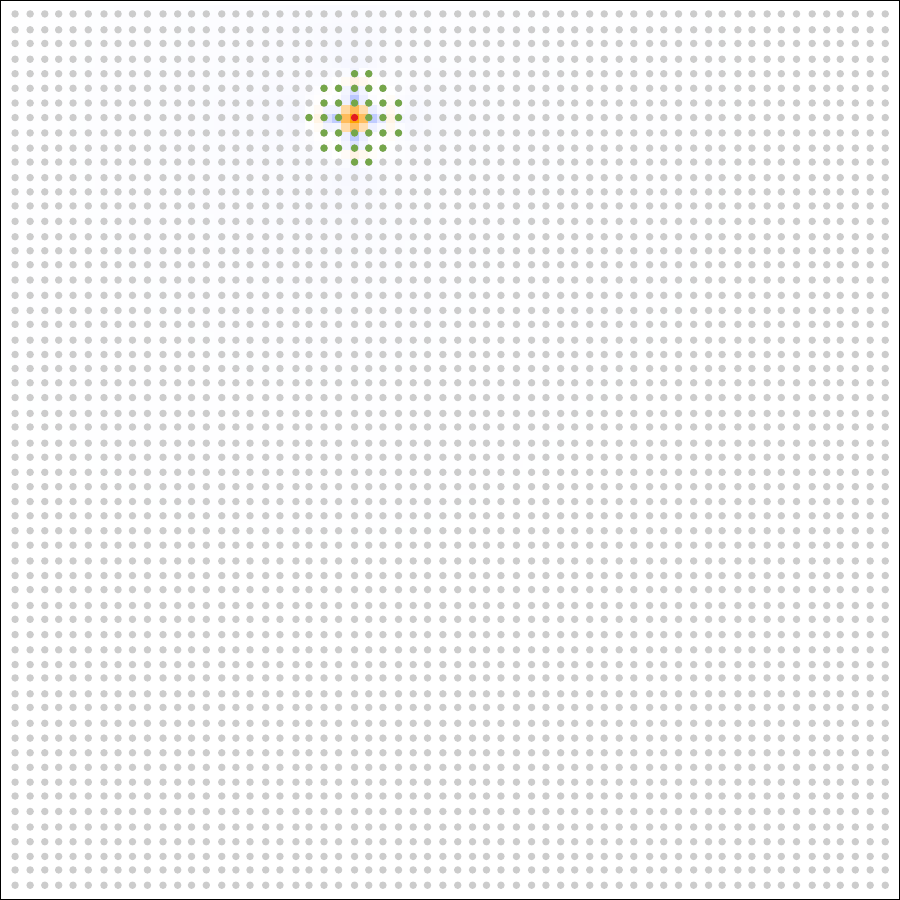}%
    \caption{For i = 1, 100, 400, 900, 1600, 2500, 3600 (from left to right), positive (orange) and negative (blue) conditional correlations with the $i$th input $\bx_i$ in the maximin ordering (red points) conditional on all (top) and $m=32$ nearest (bottom) previously ordered inputs (green points), for a GP with Mat\'ern covariance (range $r = 0.1$ and smoothness $\nu = 2$) on a grid of size $n = 60 \times 60 = 3600$. This figure is inspired by Figure 5 in \citet{Schafer2020}.}
    \label{fig:corcond_n3600}
\end{figure}

\begin{figure}[htbp]
    \centering
    \includegraphics[width=\textwidth]{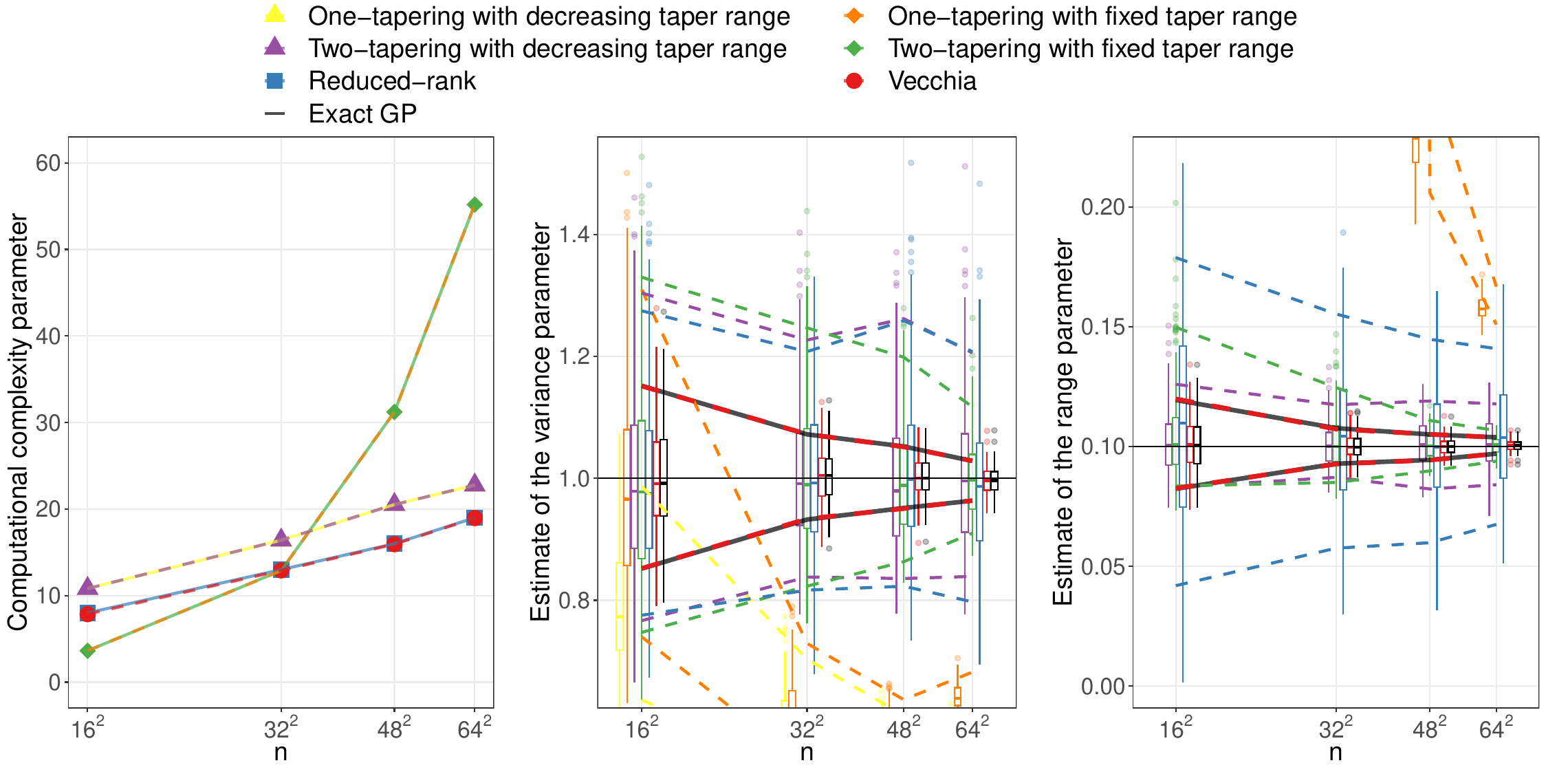}
    \caption{Comparison of maximum approximate-likelihood estimators based on taper, reduced-rank, and Vecchia approximations: The left panel compares the average number of non-zero entries of covariance matrix per observation for taper approximations, the number of inducing points for reduced-rank approximation, and the average size of conditioning set for Vecchia approximation. The center (right) panel shows box plots of estimates for the variance (range) parameter based on exact and approximate GP likelihoods for 200 synthetic datasets and 90$\%$ intervals of their sampling distributions. While estimating a parameter, the other parameter was assumed to be known. Data are generated at a regular grid of $n$ inputs on the unit square domain from a GP with Mat\'ern covariance with variance $1$, range $0.1$, and smoothness $0.5$. Note that the $x$-axes of the panels are on a log scale.}
    \label{fig:comparison_mle_boxplot}
\end{figure}


\end{document}